\documentclass[leqno]{amsart}
\usepackage[normal]{boilerplate}
\usepackage{mathrsfs}

\newcommand{\DM}{\textbf{\textup{DM}}}
\newcommand{\FEt}{\textbf{\textup{FÉt}}}
\newcommand{\FP}{\textbf{\textup{FP}}}

\newcommand{\QP}{\textbf{\textup{QP}}}

\newcommand{\Green}{\categ{Green}}
\newcommand{\Mack}{\categ{Mack}}

\newcommand{\free}{\textit{free}}

\DeclareMathOperator{\Gal}{Gal}
\DeclareMathOperator{\im}{im}
\DeclareMathOperator{\Tor}{Tor}

\newcommand{\bigboxtimes}{\operatornamewithlimits{\mathchoice{\raisebox{-0.45em}{\scalebox{2}{\ensuremath{\boxtimes}}}}{\raisebox{-0.25em}{\scalebox{1.45}{\ensuremath{\boxtimes}}}}{}{}}}

\usepackage{appendix}

\title[Spectral {M}ackey functors and equivariant algebraic {$K$}-theory {(II)}]{Spectral {M}ackey functors and equivariant algebraic {$K$}-theory {(II)}}

\author{Clark Barwick}
\address{Massachusetts Institute of Technology, Department of Mathematics, Bldg. 2, 77 Massachusetts Ave., Cambridge, MA 02139-4307, USA}
\email{clarkbar@gmail.com}

\author{Saul Glasman}
\address{School of Mathematics, Institute for Advanced Study, Fuld Hall, 1 Einstein Dr., Princeton NJ 08540}
\email{sglasman@math.ias.edu}

\author{Jay Shah}
\address{Department of Mathematics, Massachusetts Institute of Technology, 77 Massachusetts Avenue, Cambridge, MA 02139-4307, USA}
\email{jshah@math.mit.edu}

\begin{document}

\begin{abstract} We study the ``higher algebra'' of spectral Mackey functors, which the first named author introduced in Part I of this paper. In particular, armed with our new theory of symmetric promonoidal $\infty$-categories and a suitable generalization of the second named author's Day convolution, we endow the $\infty$-category of Mackey functors with a well-behaved symmetric monoidal structure. This makes it possible to speak of \emph{spectral Green functors} for any operad $O$. We also answer a question of Mathew, proving that the algebraic $K$-theory of group actions is lax symmetric monoidal. We also show that the algebraic $K$-theory of derived stacks provides an example. Finally, we give a very short, new proof of the equivariant Barratt--Priddy--Quillen theorem, which states that the algebraic $K$-theory of the category of finite $G$-sets is simply the $G$-equivariant sphere spectrum.
\end{abstract}

\maketitle

\setcounter{tocdepth}{1}
\tableofcontents


\setcounter{section}{-1}

\section{Summary} This paper is part of an effort to give a complete description of the structures available on the algebraic $K$-theory of varieties and schemes (and even of various derived stacks) with all their concomitant functorialities and homotopy coherences.

So suppose $X$ a scheme (quasicompact and quasiseparated). The derived tensor product $\otimes^{\LL}$ on perfect complexes on $X$ defines a symmetric monoidal structure on the derived category $D^{\textit{perf}}_X$ of perfect complexes on $X$. With a little more effort, one can lift this structure to a symmetric monoidal structure on the stable $\infty$-category of perfect complexes on $X$. This suffices to get a product on algebraic $K$-theory
\begin{equation*}
\otimes\colon\fromto{K(X)\wedge K(X)}{K(X)}
\end{equation*}
that is associative and commutative up to coherent homotopy. Thus, $K(X)$ has not only the structure of a connective spectrum, but also the structure of a \emph{connective $E_{\infty}$ ring spectrum}. This is an exceedingly rich structure: not only do the homotopy groups $K_{\ast}(X)$ form a graded commutative ring, but these homotopy groups also support (in a functorial way) a tremendous amount of structure involving intricate higher homotopy operations called \emph{Toda brackets}. Still more information (in the form of \emph{Dyer-Lashof operations}) can be found on the $\FF_p$-cohomology of $K(X)$.

Now for any morphism $f\colon\fromto{Y}{X}$ of schemes, the derived functor
\[\LL f^{\star}\colon\fromto{D^{\textit{qcoh}}_X}{D^{\textit{qcoh}}_Y}\] on the category of complexes with quasicoherent cohomology preserves perfect complexes, and the resulting functor $\LL f^{\star}\colon\fromto{D^{\textit{perf}}_X}{D^{\textit{perf}}_Y}$ induces a morphism
\begin{equation*}
f^{\star}\colon\fromto{K(X)}{K(Y)}
\end{equation*}
on the algebraic $K$-theory. The functor $\LL f^{\star}$ is compatible with the derived tensor product, in the sense that for any perfect complexes $E$ and $F$ on $X$, there is a canonical isomorphism
\begin{equation*}
\LL f^{\star}(E\otimes^{\LL}F)\simeq(\LL f^{\star}E)\otimes^{\LL}(\LL f^{\star}F).
\end{equation*}
Again this can be lifted to the level of stable $\infty$-categories, whence the induced morphism $f^{\star}$ on $K$-theory turns out to be a morphism of connective $E_{\infty}$ ring spectra. This implies that the induced homomorphism on homotopy groups
\begin{equation*}
f^{\star}\colon\fromto{K_{\ast}(X)}{K_{\ast}(Y)}
\end{equation*}
is a homomorphism of graded commutative rings, and it must respect all the higher homotopy operations on $K_{\ast}(X)$ as well.

Furthermore, one can fit all the functors $\LL f^{\star}$ together to get a presheaf $\goesto{U}{D^{\mathit{perf}}_U}$ on the big site of all schemes. This can even be viewed as a presheaf of stable $\infty$-categories, which suffices to give us a presheaf of connective spectra $\goesto{U}{K(U)}$. Since the morphisms $f^{\star}$ are morphisms of connective $E_{\infty}$ ring spectra, we can regard this as presheaf of $E_{\infty}$ ring spectra.

If one wanted, one might ``externalize'' the product on $K$-theory in the following manner. For any two schemes $X$ and $Y$ over a base scheme $S$, one may define an external tensor product
\[\boxtimes^{\LL}\colon\fromto{D^{\textit{perf}}_X\times D^{\textit{perf}}_X}{D^{\textit{perf}}_{X\times_SY}}\]
by the assignment $\goesto{(E,F)}{(\LL\pr_1^{\star}E)\otimes^{\LL}(\LL\pr_2^{\star}F)}$. Note that we have natural equivalences
\[(\LL f^{\star}E)\boxtimes^{\LL}(\LL g^{\star}F)\simeq\LL(f\times g)^{\star}(E\boxtimes^{\LL}F)\]
If we lift this to the level of stable $\infty$-categories, this gives rise to an external pairing
\begin{equation*}
\boxtimes\colon\fromto{K(X)\wedge K(Y)}{K(X\times_SY)},
\end{equation*}
which is functorial (contravariantly) in $X$ and $Y$. The $E_{\infty}$ product on $K(X)$ can now be obtained by pulling back this external pairing along the diagonal map:
\begin{equation*}
K(X)\wedge K(X)\to K(X\times_SX)\to K(X).
\end{equation*}

A morphism of schemes $f\colon\fromto{Y}{X}$ may induce morphisms in the \emph{covariant} direction as well. The pushforward $\RR f_{\!\star}\colon\fromto{D^{\textit{qcoh}}_Y}{D^{\textit{qcoh}}_X}$ generally will not preserve perfect complexes. If, however, $f$ is flat and proper, then for any perfect complex $E$, the complex $\RR f_{\!\star}E$ \emph{is} perfect. Thus in this case $\RR f_{\!\star}$ restricts to a functor $\RR f_{\!\star}\colon\fromto{D^{\textit{perf}}_Y}{D^{\textit{perf}}_X}$, and after lifting this to the stable $\infty$-categories, we find an induced morphism
\begin{equation*}
f_{\!\star}\colon\fromto{K(Y)}{K(X)}
\end{equation*}
on the algebraic $K$-theory. One thus obtains a covariant functor $\goesto{U}{K(U)}$, but only with respect to flat and proper morphisms. Observe, however, that since the functors $\RR f_{\!\star}$ do not commute with the derived tensor product, this functor is \emph{not} valued in ring spectra.

Nevertheless, if $f\colon\fromto{Y}{X}$ is proper and flat, we do have an algebraic structure preserved by $\RR f_{\!\star}$. Observe that one may regard $K(Y)$ as a module over the $E_{\infty}$ ring spectrum $K(X)$ via $f^{\star}$. For any perfect complexes $E$ on $Y$ and $F$ on $X$, one has a canonical equivalence
\begin{equation*}
(\RR f_{\!\star}E)\otimes^{\LL}F\simeq\RR f_{\!\star}(E\otimes^{\LL}\LL f^{\star}F)
\end{equation*}
of perfect complexes; this is the usual projection formula \cite[Exp. III, Pr. 3.7]{MR50:7133}. At the level of $K$-theory, this translates to the observation that the morphism
\begin{equation*}
f_{\!\star}\colon\fromto{K(Y)}{K(X)}
\end{equation*}
is a morphism of connective $K(X)$-modules. The induced map on homotopy groups
\begin{equation*}
f_{\!\star}\colon\fromto{K_{\ast}(Y)}{K_{\ast}(X)}
\end{equation*}
is therefore a homomorphism of $K_{\ast}(X)$-modules.

Note that the \emph{external} tensor product $\boxtimes^{\LL}$ is actually perfectly compatible with the pushforwards, in the sense that one has natural equivalences
\[(\RR f_{\!\star}E)\boxtimes^{\LL}(\RR g_{\star}F)\simeq\RR(f\times g)_\star(E\boxtimes^{\LL}F),\]
so on $K$-theory the external product $\boxtimes\colon\fromto{K(X)\wedge K(Y)}{K(X\times_SY)}$ is functorial (covariantly) in $X$ and $Y$.

Last, but certainly not least, there is a compatibility between the morphisms $f^{\star}$ and the morphisms $g_{\star}$, which results from the base change theorem for complexes \cite[Exp. IV, Pr. 3.1.0]{MR50:7133}. Suppose that
\begin{equation*}
\begin{tikzpicture}[baseline]
\matrix(m)[matrix of math nodes,
row sep=4ex, column sep=4ex,
text height=1.5ex, text depth=0.25ex]
{Y' & Y \\
X' & X \\ };
\path[>=stealth,->,font=\scriptsize]
(m-1-1) edge node[above]{$g$} (m-1-2)
edge node[left]{$f$} (m-2-1)
(m-1-2) edge node[right]{$f$} (m-2-2)
(m-2-1) edge node[below]{$g$} (m-2-2);
\end{tikzpicture}
\end{equation*}
is a pullback square of schemes in which the horizontal maps $g$ are flat and proper. Then the canonical morphism
\begin{equation*}
\fromto{\LL f^{\star}\RR g_{\star}}{\RR g_{\star}\LL f^{\star}}
\end{equation*}
is an objectwise equivalence of functors $\fromto{D^{\mathit{perf}}_{X'}}{D^{\mathit{perf}}_Y}$. This translates to the condition that there is a canonical homotopy
\begin{equation*}
f^{\star}g_{\star}\simeq g_{\star}f^{\star}\colon\fromto{K(X')}{K(Y)}
\end{equation*}
of morphisms of $K(X)$-modules. In fact, this compatibility between the pullbacks and the pushforwards, combined with the compatibility between $f_{\star}$ and the external tensor product, allows us to \emph{deduce} the projection formula.

Let us summarize the structure we've found on the assignment $\goesto{U}{K(U)}$:
\begin{itemize}
\item For every scheme $X$, we have an $E_{\infty}$ ring spectrum $K(X)$. Moreover, for any two schemes $X$ and $Y$ over a base $S$, one has an external pairing
\[\boxtimes\colon\fromto{K(X)\wedge K(Y)}{K(X\times_SY)}.\]
\item For every morphism $f\colon\fromto{Y}{X}$, we have a pullback morphism
\begin{equation*}
f^{\star}\colon\fromto{K(X)}{K(Y)},
\end{equation*}
which is compatible with the external pairings and thus also with the $E_{\infty}$ product.
\item For every flat and proper morphism $f\colon\fromto{Y}{X}$, we have a pushforward morphism
\begin{equation*}
f_{\!\star}\colon\fromto{K(Y)}{K(X)},
\end{equation*}
which is compatible with the external pairings and thus (in light of the next condition) also with the $K(X)$-module structure.
\item For any pullback square
\begin{equation*}
\begin{tikzpicture}[baseline]
\matrix(m)[matrix of math nodes,
row sep=4ex, column sep=4ex,
text height=1.5ex, text depth=0.25ex]
{Y' & Y \\
X' & X \\ };
\path[>=stealth,->,font=\scriptsize]
(m-1-1) edge node[above]{$g$} (m-1-2)
edge node[left]{$f$} (m-2-1)
(m-1-2) edge node[right]{$f$} (m-2-2)
(m-2-1) edge node[below]{$g$} (m-2-2);
\end{tikzpicture}
\end{equation*}
in which the horizontal maps $g$ are flat and proper, we have a canonical homotopy
\begin{equation*}
f^{\star}g_{\star}\simeq g_{\star}f^{\star}\colon\fromto{K(X')}{K(Y)}.
\end{equation*}
of morphisms of $K(X)$-modules.
\end{itemize}
In this paper, we will demonstrate that these structures, along with all of their homotopy coherences, are neatly packaged in a \emph{spectral Green functor} on the category of schemes.

This structure is the origin of both the $\Gal(E/F)$-equivariant $E_\infty$ ring spectrum structure on the algebraic $K$-theory of a Galois extension $E\supset F$ and the cyclotomic structure on the $p$-typical curves on a smooth $\FF_p$-scheme. For the former, see \ref{cnstr:Galoisequivarianteinftyalg}, and for the latter, see the forthcoming paper \cite{cyclon}.

In order to describe all the structure we see here, we study the ``higher algebra'' (in the sense of Lurie's book \cite{HA}, for example) of spectral Mackey functors, which we introduced in Part I of this paper \cite{M1}. The $\infty$-category of spectral Mackey functors turns out to admit all the same well-behaved structures as the $\infty$-category of spectra itself. In particular, the $\infty$-category of Mackey functors admits a well-behaved symmetric monoidal structure. This, combined with Saul Glasman's  convolution for $\infty$-categories \cite{GlasmanDay}, makes it possible to speak of $E_1$ algebras, $E_\infty$ algebras, or indeed $O$-algebras for any operad $O$ in this context; these are called \emph{$O$-Green functors}.

We use this framework to provide a very simple answer to a question posed to us by Akhil Mathew, in which we demonstrate that the functor that assigns to any $\infty$-category with an action of a finite group $G$ its equivariant algebraic $K$-theory is lax symmetric monoidal. We also show that the algebraic $K$-theory of derived stacks with its transfer maps as described above offers an example of an $E_{\infty}$ Green functor. We also use this theory to give a new proof of the equivariant Barratt--Priddy--Quillen theorem, which states that the algebraic $K$-theory of the category of finite $G$-sets is simply the $G$-equivariant sphere spectrum. (In fact, we will generalize this result dramatically.)

\subsection*{Warning} Let us emphasize that $E_{\infty}$-Green functors for a finite group $G$ are \emph{not} equivalent to algebras in $G$-equivariant spectra structured by the equivariant linear isometries operad on a complete $G$-universe. To describe the latter in line with the discussion here -- and to find such structures on algebraic $K$-theory spectra -- it is necessary to develop elements of the theory of $G$-$\infty$-categories. This we do in the forthcoming joint paper \cite{BDGNS2}.

\subsection*{Acknowledgments} We have had very helpful conversations with David Ayala and Mike Hill about the contents of this paper, its predecessor, and its sequels. We also thank the other participants of the Bourbon Seminar -- Emanuele Dotto, Marc Hoyois, Denis Nardin, and Tomer Schlank -- for their many, many insights.




\section{$\infty$-anti-operads and symmetric promonoidal $\infty$-categories} One of the many complications that arises when one combines an $\infty$-category and its opposite in the way we have in our construction of the effective Burnside $\infty$-category is that our constructions are extremely intolerant of asymmetries in basic definitions. This complication rears its head the moment we want to contemplate the symmetric monoidal structure on the Burnside $\infty$-category. In effect, the description of a symmetric monoidal $\infty$-categories given in \cite[Ch. 4]{HA} forces one to specify the data of maps \emph{out of} various tensor products in a suitably compatible fashion. Thus symmetric monoidal categories are there identified as certain \emph{$\infty$-operads}. But since we are also working with opposites of symmetric monoidal $\infty$-categories, we will come face-to-face with circumstances in which we must identify the data of maps \emph{into} various tensor products in a suitably compatible fashion. We will call the resulting opposites of $\infty$-operads \emph{$\infty$-anti-operads}.\footnote{We do not know a standard name for this structure. In a previous verion of this paper, CB called these ``cooperads,'' but this conflicts with better-known terminology.} Awkward as this may seem, it cannot be avoided.

\begin{ntn} Let $\Lambdaup(\FF)$ denote the following ordinary category. The objects will be finite sets, and a morphism $\fromto{J}{I}$ will be a map $\fromto{J}{I_{+}}$; one composes $\psi\colon\fromto{K}{J_{+}}$ with $\phi\colon\fromto{J}{I_{+}}$ by forming the composite
\begin{equation*}
K\ \tikz[baseline]\draw[>=stealth,->,font=\scriptsize](0,0.5ex)--node[above]{$\psi$}(0.5,0.5ex);\ J_{+}\ \tikz[baseline]\draw[>=stealth,->,font=\scriptsize](0,0.5ex)--node[above]{$\phi_{+}$}(0.5,0.5ex);\ I_{++}\ \tikz[baseline]\draw[>=stealth,->,font=\scriptsize](0,0.5ex)--node[above]{$\mu$}(0.5,0.5ex);\ I_{+},
\end{equation*}
where $\mu\colon\fromto{I_{++}}{I_{+}}$ is the map that simply identifies the two added points. (Of course $\Lambdaup(\FF)$ is equivalent to the category $\FF_{\ast}$ of pointed finite sets, but we prefer to think of the objects of $\Lambdaup(\FF)$ as unpointed. This is the natural perspective on this category from the theory of operator categories \cite{opcat}.)
\end{ntn}

\begin{dfn}\label{dfn:anti-operad}
\begin{enumerate}[(\ref{dfn:anti-operad}.1)]
\item An \textbf{\emph{$\infty$-anti-operad}} is an inner fibration
\begin{equation*}
p\colon\fromto{O_{\otimes}}{\N\Lambdaup(\FF)^{\op}}
\end{equation*}
whose opposite
\begin{equation*}
p^{\op}\colon\fromto{(O_{\otimes})^{\op}}{\N\Lambdaup(\FF)}
\end{equation*}
is an $\infty$-operad.
\item If $p\colon\fromto{O_{\otimes}}{\N\Lambdaup(\FF)^{\op}}$ is an $\infty$-anti-operad, then an edge of $O_{\otimes}$ will be said to be \textbf{\emph{inert}} if it is cartesian over an edge of $\N\Lambdaup(\FF)^{\op}$ that corresponds to an inert map in $\Lambdaup(\FF)$, that is, a map $\phi\colon\fromto{J}{I_{+}}$ such that the induced map $\fromto{\phi^{-1}(I)}{I}$ is a bijection \cite[Df. 2.1.1.8]{HA}, \cite[Df. 8.1]{opcat}.
\item A cartesian fibration
\begin{equation*}
q\colon\fromto{X_{\otimes}}{O_{\otimes}}
\end{equation*}
will be said to \textbf{\emph{exhibit $X_{\otimes}$ as an $O_{\otimes}$-monoidal $\infty$-category}} just in case the cocartesian fibration
\begin{equation*}
q^{\op}\colon\fromto{(X_{\otimes})^{\op}}{(O_{\otimes})^{\op}}
\end{equation*}
exhibits $(X_{\otimes})^{\op}$ as an $(O_{\otimes})^{\op}$-monoidal $\infty$-category in the sense of \cite[Df. 2.1.2.13]{HA}. When $O_{\otimes}=\N\Lambdaup(\FF)^{\op}$, we will say that $q$ \textbf{\emph{exhibits $X_{\otimes}$ as a symmetric monoidal $\infty$-category}}.
\item A \textbf{\emph{morphism $f\colon\fromto{O_{\otimes}}{P_{\otimes}}$ of $\infty$-anti-operads}} is a morphism over $\N\Lambdaup(\FF)^{\op}$ that carries inert edges to inert edges. If $O_{\otimes}$ and $P_{\otimes}$ are symmetric monoidal $\infty$-categories, then $f$ is a \textbf{\emph{symmetric monoidal functor}} if it carries all cartesian edges to cartesian edges.
\end{enumerate}
\end{dfn}

\begin{exm} Suppose $C$ an $\infty$-category. We define the \textbf{\emph{cartesian $\infty$-anti-operad}} as
\begin{equation*}
p\colon\fromto{C_{\times}\coloneq((C^{\op})^{\sqcup})^{\op}}{\N\Lambdaup(\FF)^{\op}},
\end{equation*}
where the notation $(\cdot)^{\sqcup}$ refers to the cocartesian $\infty$-operad \cite[Cnstr. 2.4.3.1]{HA}.  If $C$ is an $\infty$-category that admits all products, then the functor $p$ exhibits $C_{\times}$ as a symmetric monoidal $\infty$-category \cite[Rk. 2.4.3.4]{HA}.

An object $(I,X)$ of $C_{\times}$ consists of a finite set $I$ and a family $\{X_i\ |\ i\in I\}$; a morphism $(\phi,\omega)\colon\fromto{(I,X)}{(J,Y)}$ of $C_{\times}$ consists of a map of finite sets $\phi\colon\fromto{J}{I_{+}}$ and a family of morphisms
\begin{equation*}
\left\{\omega_j\colon\fromto{X_{\phi(j)}}{Y_j}\;\middle|\;j\in \phi^{-1}(I)\right\}
\end{equation*}
of $C$. If $C$ admits finite products, then the morphisms $\omega_j$ determine and are determined by a family of morphisms
\begin{equation*}
\left\{\ \omega_{J_i}\colon\fromto{X_i}{\prod_{j\in J_i}Y_j}\quad\middle|\quad i\in I\ \right\};
\end{equation*}
here $J_i$ denotes the fiber $\phi^{-1}(i)$.

Observe that the cartesian $\infty$-anti-operad is significantly simpler to define than the cartesian $\infty$-operad. Note also that $(\Deltaup^0)_{\times}=\N\Lambdaup(\FF)^{\op}$.
\end{exm}

It is extremely useful to note that the condition that an $\infty$-operad $C^{\otimes}$ be a symmetric monoidal $\infty$-category can be broken into two conditions:
\begin{enumerate}[(1)]
\item The first of these is \emph{corepresentability} \cite[Df. 6.2.4.3]{HA}; this is the condition that the functors $\Map^{\xi_I}_{C^{\otimes}}(x_I,-)\colon\fromto{C}{\Top}$ be corepresentable, where $\xi_I$ is the unique active map $\fromto{I}{\ast}$ in $\Lambdaup(\FF)$. A compact expression of this is simply to say (as Lurie does) that the inner fibration $\fromto{C^{\otimes}}{N\Lambdaup(\FF)}$ is locally cocartesian.
\item The second condition is \emph{symmetric promonoidality}. This can be expressed in a number of ways. One may say that for any active map $\phi\colon\fromto{J}{I}$ of $\Lambdaup(\FF)$, for any object $x_J\in C^{\otimes}_J$, and for any object $z\in C$, the natural map
\[\fromto{\int^{y_I\in C^{\otimes}_I} \Map^{\xi_I}_{C^{\otimes}}(y_I,z)\times\Map^{\phi}_{C^{\otimes}}(x_J, y_I)}{\Map^{\xi_J}_{C^{\otimes}}(x_J, z)}\]
is an equivalence; this is an operadic version of the condition expressed in \cite[Ex. 6.2.4.9]{HA}. Equivalently, $C^{\otimes}$ is a symmetric promonoidal $\infty$-category if it represents a commutative algebra object in the $\infty$-category of $\infty$-categories and profunctors. In light of \cite[\S B.3]{HA}, we make the following definition.
\end{enumerate}

\begin{dfn} We will say that an $\infty$-operad $C^{\otimes}$ is \textbf{\emph{symmetric promonoidal}} if the structure map $\fromto{C^{\otimes}}{N\Lambdaup(\FF)}$ is a flat inner fibration \cite[Df. B.3.8]{HA}. Similarly, we will say that an $\infty$-anti-operad $C_{\otimes}$ is \textbf{\emph{symmetric promonoidal}} if the structure map $\fromto{C_{\otimes}}{N\Lambdaup(\FF)^{\op}}$ is a flat inner fibration.
\end{dfn}

Our claim now is that the conjunction of these two conditions are equivalent to the condition that $C^{\otimes}$ be a symmetric monoidal $\infty$-category. That is, we claim that a symmetric monoidal $\infty$-category is \emph{precisely} a corepresentable symmetric promonoidal $\infty$-category. This follows immediately from the following.
\begin{prp}\label{prp:flatloccocartiscocart} The following are equivalent for an inner fibration $p\colon\fromto{X}{S}$.
\begin{enumerate}[(\ref{prp:flatloccocartiscocart}.1)]
\item The inner fibration $p$ is flat and locally cocartesian.
\item The inner fibration $p$ is cocartesian.
\end{enumerate}
\begin{proof} The second condition implies the first by \cite[Ex. B.3.11]{HA}. Let us show that the first condition implies the second. By \cite[Pr. 2.4.2.8]{HTT}, it suffices to consider the case in which $S=\Deltaup^2$, and to show that for any section of $p$ given by a commutative triangle
\begin{equation*}
\begin{tikzpicture}[baseline]
\matrix(m)[matrix of math nodes,
row sep=4ex, column sep=4ex,
text height=1.5ex, text depth=0.25ex]
{ &y&  \\
x && z \\ };
\path[>=stealth,->,font=\scriptsize]
(m-2-1) edge[inner sep=0.5pt] node[above left]{$f$} (m-1-2)
edge node[below]{$h$} (m-2-3)
(m-1-2) edge[inner sep=0.5pt] node[above right]{$g$} (m-2-3);
\end{tikzpicture}
\end{equation*}
in which $f$ and $g$ are locally $p$-cocartesian, the edge $h$ is locally $p$-cocartesian as well.

In this case, by \cite[Cor. 3.3.1.2]{HTT}, we can find a cocartesian fibration $q\colon\fromto{Y}{\Deltaup^2}$ along with an equivalence
\[\phi\colon\equivto{X \times_{\Deltaup^2} \Lambdaup^2_1}{Y \times_{\Deltaup^2} \Lambdaup^2_1}\]
of cocartesian fibrations over $\Lambda^2_1$. Now since $p$ is flat, the inclusion $\into{X \times_{\Deltaup^2} \Lambdaup^2_1}{X}$ is a categorical equivalence over $\Deltaup^2$. Consequently, we may lift to obtain a map $\psi\colon\fromto{X}{Y}$ over $\Deltaup^2$ extending $\phi$. This map is a categorical equivalence since both $p$ and $q$ are flat.

Now $\psi(f)=\phi(f)$ and $\psi(g)=\phi(g)$ are $q$-cocartesian, whence so is $\psi(h)$. The stability of relative colimits under categorical equivalences \cite[Pr. 4.3.1.6]{HTT}, in light of \cite[Ex. 4.3.1.4]{HTT}, now implies that $h$ is $p$-cocartesian.
\end{proof}
\end{prp}

One reason to treasure symmetric promonoidal structures is the fact that, as we shall now prove, they are precisely the structure needed on an $\infty$-category $C$ in order for $\Fun(C,D)$ to admit a \emph{Day convolution} symmetric monoidal structure.\footnote{We would like to acknowledge that Dylan Wilson has independently made this observation.}

To explain, suppose first $C^{\otimes}$ a small symmetric monoidal $\infty$-category, and suppose $D^{\otimes}$ a symmetric monoidal $\infty$-category such that $D$ admits all colimits, and the tensor product preserves colimits separately in each variable. In \cite{GlasmanDay}, Glasman constructs a symmetric monoidal structure on the functor $\infty$-category $\Fun(C,D)$ which is the natural $\infty$-categorical generalization of Day's convolution product. As in Day's construction, the convolution $F\otimes G$ of two functors $F,G\colon\fromto{C}{D}$ in Glasman's symmetric monoidal structure is given by the left Kan extension of the composite
\begin{equation*}
C\times C\ \tikz[baseline]\draw[>=stealth,->,font=\scriptsize](0,0.5ex)--node[above]{$(F,G)$}(0.9,0.5ex);\ D\times D\ \tikz[baseline]\draw[>=stealth,->,font=\scriptsize](0,0.5ex)--node[above]{$\otimes$}(0.65,0.5ex);\ D
\end{equation*}
along the tensor product $\otimes\colon\fromto{C\times C}{C}$.

In particular, for any finite set $I$, and for any $I$-tuple $\{F_i\}_{i\in I}$ of functors $\fromto{C}{D}$, the value of the tensor product is given by the coend
\[\left(\bigotimes_{i\in I}F_i\right)(x)\simeq\int^{u_I\in C^{\otimes}_I}\Map^{\xi_I}_{C^{\otimes}}(u_I,x)\otimes\bigotimes_{i\in I}F_i(u_i).\]

Equivalently, the Day convolution on $\Fun(C,D)$ is the essentially unique symmetric monoidal structure that enjoys the following criteria:
\begin{itemize}
\item The tensor product
\[-\otimes-\colon\fromto{\Fun(C,D)\times\Fun(C,D)}{\Fun(C,D)}\]
preserves colimits separately in each variable.
\item The functor given by the composite
\[C^{\op}\times D\ \tikz[baseline]\draw[>=stealth,->,font=\scriptsize](0,0.5ex)--node[above]{$j\times\id$}(0.75,0.5ex);\ \Fun(C,\Kan)\times D\ \tikz[baseline]\draw[>=stealth,->,font=\scriptsize](0,0.5ex)--node[above]{$m$}(0.5,0.5ex);\ \Fun(C,D)\]
is symmetric monoidal, where $j$ denotes the Yoneda embedding, and $m$ is the functor corresponding to the composition
\[\Fun(C,\Kan)\to\Fun(D\times C,D\times\Kan)\to\Fun(D\times C,D)\]
in which the first functor is the obvious one, and the functor $\fromto{D\times\Kan}{D}$ is the tensor functor $\goesto{(X,K)}{X\otimes K}$ of \cite[\S 4.4.4]{HTT}.
\end{itemize}

Conveniently, we can extend Glasman's Day convolution to situations in which $C^{\otimes}$ is only symmetric promonoidal.
\begin{prp} For any symmetric promonoidal $\infty$-category $C^{\otimes}$ and any symmetric monoidal $\infty$-category $D^{\otimes}$ such that $D$ admits all colimits and $\otimes\colon\fromto{D\times D}{D}$ preserves colimits separately in each variable, $\Fun(C,D)$ admits a symmetric monoidal structure such that the $E_{\infty}$-algebras therein are morphisms of $\infty$-operads $\fromto{C^{\otimes}}{D^{\otimes}}$.
\begin{proof} The results of the first two sections of \cite{GlasmanDay} hold when $C^{\otimes}$ is symmetric promonoidal with only one change: in the proof of \cite[Lm. 2.3]{GlasmanDay}, the reference to \cite[Pr. 3.3.1.3]{HTT} should be replaced with a reference to \cite[Pr. B.3.14]{HA}. Consequently, our claim follows from \cite[Prs. 2.11 and 2.12]{GlasmanDay}.
\end{proof}
\end{prp}

\begin{nul} Once again, for any finite set $I$, and for any $I$-tuple $\{F_i\}_{i\in I}$ of functors $\fromto{C}{D}$, the value of the tensor product is given by the coend
\[\left(\bigotimes_{i\in I}F_i\right)(x)\simeq\int^{u_I\in C^{\otimes}_I}\Map^{\xi_I}_{C^{\otimes}}(u_I,x)\otimes\bigotimes_{i\in I}F_i(u_i).\]
\end{nul}


\section{The symmetric promonoidal structure on the effective Burnside $\infty$-category} Suppose $C$ a disjunctive $\infty$-category. The product on $C$ does not induce the product on the effective Burnside $\infty$-category $A^{\eff}(C)$. (Indeed, recall that the effective Burnside $\infty$-category admits direct sums, and these direct sums are induced by the \emph{coproduct} in $C$.) However, a product on $C$ (if it exists) \emph{does} induce a symmetric monoidal structure on $A^{\eff}(C)$. The construction of the previous example is just what we need to describe this structure, and it will work for a broad class of disjunctive triples -- which we call \emph{cartesian} -- as well.

It turns out to be convenient to consider situations in which $C$ does not actually have products. In this case, the effective Burnside $\infty$-category $A^{\eff}(C)$ admits not a symmetric monoidal structure, but only a symmetric promonoidal structure, which suffices to get the Day convolution on $\infty$-categories of Mackey functors.

\begin{ntn}\label{ntn:triplestronCtimes} Suppose $(C,C_{\dag},C^{\dag})$ a disjunctive triple. We now define a triple structure $(C_{\times},(C_{\times})_{\dag},(C_{\times})^{\dag})$ on $C_{\times}$ in the following manner. A morphism
\begin{equation*}
(\phi,\omega)\colon\fromto{(I,X)}{(J,Y)}
\end{equation*}
of $C_{\times}$ will be ingressive just in case $\phi$ is a bijection, and each morphism
\begin{equation*}
\omega_j\colon\fromto{X_{\phi(j)}}{Y_j}
\end{equation*}
is ingressive. The morphism $(\phi,\omega)$ will be egressive just in case each morphism
\begin{equation*}
\omega_j\colon\fromto{X_{\phi(j)}}{Y_j}
\end{equation*}
is egressive (with no condition on $\phi$).
\end{ntn}

It is a trivial matter to verify the following.
\begin{lem} Suppose $(C,C_{\dag},C^{\dag})$ a left complete disjunctive triple. Then the triple
\begin{equation*}
(C_{\times},(C_{\times})_{\dag},(C_{\times})^{\dag})
\end{equation*}
is adequate in the sense of \cite[Df. 5.2]{M1}.
\end{lem}

In particular, for any left complete disjunctive triple $(C,C_{\dag},C^{\dag})$, one may consider the effective Burnside $\infty$-category
\[A^{\eff}(C_{\times},(C_{\times})_{\dag},(C_{\times})^{\dag}).\]

\begin{exm}\label{exm:AeffofDeltazeroisLambdaFop} Note in particular that
\begin{equation*}
((\Deltaup^0)_{\times},((\Deltaup^0)_{\times})_{\dag},((\Deltaup^0)_{\times})^{\dag})\simeq(\N\Lambdaup(\FF)^{\op},\iota \N\Lambdaup(\FF)^{\op},\N\Lambdaup(\FF)^{\op}),
\end{equation*}
whence one proves easily that the inclusion
\begin{equation*}
\into{\N\Lambdaup(\FF) \simeq(((\Deltaup^0)_{\times})^{\dag})^{\op}}{A^{\eff}((\Deltaup^0)_{\times},((\Deltaup^0)_{\times})_{\dag},((\Deltaup^0)_{\times})^{\dag})}
\end{equation*}
is an equivalence.
\end{exm}

We'll use the following pair of results. They follow the same basic pattern as \cite[Lms. 11.4 and 11.5]{M1}; in particular, they too follow immediately from the first author's ``omnibus theorem'' \cite[Th. 12.2]{M1}.

\begin{lem}\label{lm:ACCtoNLFisinner} Suppose $(C,C_{\dag},C^{\dag})$ a left complete disjunctive triple. Then the natural functor
\begin{equation*}
\fromto{A^{\eff}(C_{\times},(C_{\times})_{\dag},(C_{\times})^{\dag})}{A^{\eff}((\Deltaup^0)_{\times},((\Deltaup^0)_{\times})_{\dag},((\Deltaup^0)_{\times})^{\dag})}
\end{equation*}
is an inner fibration.
\end{lem}

\begin{lem}\label{lm:ACCtoNLFiscocart} Suppose $(C,C_{\dag},C^{\dag})$ a left complete disjunctive triple. Then for any object $Y$ of $C_{\times}$ lying over an object $J\in(\Deltaup^0)_{\times}$ and any inert morphism $\phi\colon\fromto{I}{J}$ of $N\Lambda(\FF)$, there exists a cocartesian edge $\fromto{Y}{X}$ for the inner fibration
\begin{equation*}
\fromto{A^{\eff}(C_{\times},(C_{\times})_{\dag},(C_{\times})^{\dag})}{A^{\eff}((\Deltaup^0)_{\times},((\Deltaup^0)_{\times})_{\dag},((\Deltaup^0)_{\times})^{\dag})}
\end{equation*}
lying over the image of $\phi$ under the equivalence of Ex. \ref{exm:AeffofDeltazeroisLambdaFop}.
\end{lem}



Now we can go about defining the symmetric promonoidal structure on the effective Burnside $\infty$-category of a disjunctive triple.
\begin{ntn} For any disjunctive triple $(C,C_{\dag},C^{\dag})$, we define $A^{\eff}(C,C_{\dag},C^{\dag})^{\otimes}$ as the pullback
\begin{equation*}
A^{\eff}(C,C_{\dag},C^{\dag})^{\otimes}\coloneq A^{\eff}(C_{\times},(C_{\times})_{\dag},(C_{\times})^{\dag})\times_{A^{\eff}((\Deltaup^0)_{\times},((\Deltaup^0)_{\times})_{\dag},((\Deltaup^0)_{\times})^{\dag})}\N\Lambdaup(\FF),
\end{equation*}
equipped with its canonical projection to $\N\Lambdaup(\FF)$. Note that because the inclusion
\[\into{\N\Lambdaup(\FF)}{A^{\eff}((\Deltaup^0)_{\times},(\Deltaup^0)_{\times,\dag},(\Deltaup^0)_{\times}^{\dag})}\]
is an equivalence, it follows that the projection functor
\begin{equation*}
\fromto{A^{\eff}(C,C_{\dag},C^{\dag})^{\otimes}}{A^{\eff}(C_{\times},(C_{\times})_{\dag},(C_{\times})^{\dag})}
\end{equation*}
is actually an equivalence.
\end{ntn}

\begin{rem}\label{rem:explicitAeffotimes} Suppose $(C,C_{\dag},C^{\dag})$ a disjunctive triple. The objects of the total $\infty$-category $A^{\eff}(C,C_{\dag},C^{\dag})^{\otimes}$ are pairs $(I,X_I)$ consisting of a finite set $I$ and an $I$-tuple $X_I=(X_i)_{i\in I}$ of objects of $C$. A morphism
\begin{equation*}
\fromto{(J,Y_J)}{(I,X_I)}
\end{equation*}
of $A^{\eff}(C,C_{\dag},C^{\dag})^{\otimes}$ can be thought of as a morphism $\phi\colon\fromto{J}{I}$ of $\Lambdaup(\FF)$ and a collection of diagrams
\begin{equation*}
\left\{\ 
\begin{tikzpicture}[baseline]
\matrix(m)[matrix of math nodes, 
row sep={7ex,between origins}, column sep={7ex,between origins}, 
text height=1.5ex, text depth=0.25ex] 
{&U_{\phi(j)}&\\ 
Y_{j}&&X_{\phi(j)},\\}; 
\path[>=stealth,->,font=\scriptsize] 
(m-1-2) edge[->>] (m-2-1) 
edge[>->] (m-2-3); 
\end{tikzpicture}
\quad\middle|\quad j\in \phi^{-1}(I)\ 
\ \right\}
\end{equation*}
such that for any $j\in J$, the morphism $\cofto{U_{\phi(j)}}{X_{\phi(j)}}$ is ingressive, and the morphism
\begin{equation*}
\fibto{U_{\phi(j)}}{Y_j}
\end{equation*}
is egressive.

Composition is then defined by pullback; that is, a $2$-simplex
\begin{equation*}
(K,Z_K)\to (J,Y_J)\to (I,X_I)
\end{equation*}
consists of morphisms $\psi\colon\fromto{K}{J}$ and $\phi\colon\fromto{J}{I}$ of $\Lambdaup(\FF)$ along with a collection of diagrams
\begin{equation*}
\left\{\ 
\begin{tikzpicture}[baseline]
\matrix(m)[matrix of math nodes, 
row sep={7ex,between origins}, column sep={7ex,between origins}, 
text height=1.5ex, text depth=0.25ex] 
{&&W_{\phi(\psi(k))}&&\\
&V_{\psi(k)}&&U_{\phi(\psi(k))}&\\ 
Z_{k}&&Y_{\psi(k)}&&X_{\phi(\psi(k))}\\}; 
\path[>=stealth,->,font=\scriptsize] 
(m-1-3) edge[->>] (m-2-2) 
edge[>->] (m-2-4) 
(m-2-2) edge[->>] (m-3-1)
edge[>->] (m-3-3) 
(m-2-4) edge[->>] (m-3-3)
edge[>->] (m-3-5); 
\end{tikzpicture}
\quad\middle|\quad k\in(\phi\psi)^{-1}(I)\ 
\ \right\}
\end{equation*}
in which the square in the middle exhibits each $W_i$ (for $i\in I$) as the iterated fiber product over $U_i$ of the set of objects $\{V_j\times_{Y_j}U_i\ |\ j\in J_i\}$. (Note that the left completeness is used to show that this iterated fiber product exists.)

In particular, $A^{\eff}(C,C_{\dag},C^{\dag})^{\otimes}_{\{1\}}$ may be identified with the effective Burnside $\infty$-category $A^{\eff}(C,C_{\dag},C^{\dag})$ itself, and for any finite set $I$, the inert morphisms $\chi_{i}\colon\fromto{I}{\{i\}_{+}}$ together induce an equivalence
\begin{equation*}
\equivto{A^{\eff}(C,C_{\dag},C^{\dag})^{\otimes}_I}{\prod_{i\in I}A^{\eff}(C,C_{\dag},C^{\dag})^{\otimes}_{\{i\}}}.
\end{equation*}
\end{rem}

For the proofs of the next few results it is convenient to introduce a bit of notation.

\begin{ntn} Suppose $(C,C_{\dag},C^{\dag})$ a triple, suppose $A$ and $B$ are two sets, and suppose $S\colon\fromto{A\sqcup B}{C}$ a functor. Then let
\[C'_{/\{S_{x}\ ;\ S_{y}\}_{x\in A, y\in B}}\subseteq C_{/\{S_{z}\}_{z\in A\sqcup B}}\]
denote the full subcategory spanned by those objects such that the morphisms to the objects $S_x$ are egressive and the morphisms to the objects $S_y$ are ingressive. In particular, note that
\[\Map_{A^{\eff}(C,C_{\dag},C^{\dag})^{\otimes}}((J,Y_J),(\ast,X))\simeq\iota C'_{/\{Y_j\ ;\ X\}_{j\in J}}.\]
\end{ntn}

We have almost proven the following.
\begin{prp} For any left complete disjunctive triple $(C,C_{\dag},C^{\dag})$, the inner fibration 
\begin{equation*}
\fromto{A^{\eff}(C,C_{\dag},C^{\dag})^{\otimes}}{\N\Lambdaup(\FF)}
\end{equation*}
is an $\infty$-operad.
\end{prp}
\begin{proof} Following Rk. \ref{rem:explicitAeffotimes}, it only remains to show that given an edge $\alpha \colon \fromto{I}{J}$ in $\N\Lambdaup(\FF)$ and objects $(I,X), (J,Y)$ in $A^{\eff}(C,C_{\dag},C^{\dag})^{\otimes}$, the cocartesian edges 

\begin{equation*}
\begin{tikzpicture}[baseline]
\matrix(m)[matrix of math nodes, 
row sep={7ex,between origins}, column sep={7ex,between origins}, 
text height=1.5ex, text depth=0.25ex] 
{ &(\ast,Y_j)& \\ 
(J,Y) && (\ast,Y_j),\\}; 
\path[>=stealth,->,font=\scriptsize] 
(m-1-2) edge[-,double distance=1.5pt] (m-2-1) 
edge[>->] (m-2-3); 
\end{tikzpicture}
\end{equation*}
over the inert edges $\rho^j \colon \fromto{J}{\ast}$ induce an equivalence 

\[\fromto{\Map_{A^{\eff}(C,C_{\dag},C^{\dag})^{\otimes}}^{\alpha} ((I,X),(J,Y))}{\prod_{j \in J} \Map_{A^{\eff}(C,C_{\dag},C^{\dag})^{\otimes}}^{\rho^j \circ \alpha} ((I,X),(\ast,Y_j))}. \]
But this is indeed true, since the map identifies the left-hand side as
\[ \prod_{j \in J} \iota C'_{/\{ X_i\ ;\ Y_j \}_{i \in \alpha^{-1}(j)} }.\qedhere \]
\end{proof}

We now show that the $\infty$-operad $A^{\eff}(C,C_{\dag},C^{\dag})^{\otimes}$ is symmetric promonoidal.

\begin{prp} \label{lm:ACCtoNLFisflat} Suppose $(C,C_{\dag},C^{\dag})$ a left complete disjunctive triple. Then the $\infty$-operad 
\begin{equation*}
p\colon\fromto{A^{\eff}(C,C_{\dag},C^{\dag})^{\otimes}}{\N\Lambdaup(\FF)}
\end{equation*}
is symmetric promonoidal; that is, $p$ is a flat inner fibration.
\begin{proof} Suppose $\sigma\colon\fromto{\Delta^2}{N\Lambdaup(\FF)}$ a $2$-simplex given by a diagram
\begin{equation*}
\begin{tikzpicture}[baseline]
\matrix(m)[matrix of math nodes,
row sep=4ex, column sep=4ex,
text height=1.5ex, text depth=0.25ex]
{ &J&  \\
I && K \\ };
\path[>=stealth,->,font=\scriptsize]
(m-2-1) edge[inner sep=0.65pt] node[above left]{$ \alpha $} (m-1-2)
edge node[below]{$ \gamma $} (m-2-3)
(m-1-2) edge[inner sep=0.65pt] node[above right]{$ \beta$} (m-2-3);
\end{tikzpicture}
\end{equation*}
a $2$-simplex of $N\Lambdaup(\FF)$. Suppose 
\begin{equation*}
\begin{tikzpicture}[baseline]
\matrix(m)[matrix of math nodes, 
row sep={7ex,between origins}, column sep={7ex,between origins}, 
text height=1.5ex, text depth=0.25ex] 
{ &(K,W)& \\ 
(I,X) && (K,Z),\\}; 
\path[>=stealth,->,font=\scriptsize] 
(m-1-2) edge[->>] (m-2-1) 
edge[>->] (m-2-3); 
\end{tikzpicture}
\end{equation*}
an edge $\widetilde{\gamma}$ of
\[\Sigma\coloneq A^{\eff}(C,C_{\dag},C^{\dag})^{\otimes}\times_{N\Lambdaup(\FF),\sigma}\Delta^2\]
lifting $\gamma$. Set 
\begin{equation*}
E \coloneq \Sigma_{(I,X)/\ /(K,Z)} \times_{N\Lambdaup(\FF)} \{ J \}
\end{equation*}
be the $\infty$-category of factorizations of $\widetilde{\gamma}$ through $\Sigma_J$. Observe that an $n$-simplex of $E$ is a cartesian functor $\widetilde{\mathscr{O}}(\Delta^{n+2})^\op \to (C_{\times},(C_{\times})_{\dag},(C_{\times})^{\dag})$ satisfying certain conditions.

We aim to show that $E$ is weakly contractible. To this end, we will identify a full subcategory $E'\subset E$ whose inclusion functor admits a right adjoint such that $E'$ contains a terminal object.

To begin, let us define a functor $\epsilon\colon E \times \Delta^1 \to E$ extending the projection $\equivto{E \times \{1 \}}{E}$ as follows: given non-negative integers $k \leq n$, let
\[f_{n,k}\colon \widetilde{\mathscr{O}}(\Delta^{n+3}) \to \widetilde{\mathscr{O}}(\Delta^{n+2})\]
be the unique functor which on objects is given by
\begin{equation*}
f_{n,k}(i j) \coloneq\begin{cases}
0 j & \textup{if }i \leq k+1\textup{ and }j \leq k+1;\\
0 (j-1) & \textup{if }i \leq k+1\textup{ and }j > k+1; \\
(i-1) (j-1) & \textup{if } i > k+1.
\end{cases}
\end{equation*} 
Then for every $n$-simplex $\sigma\colon \Delta^n \to E$ corresponding to a functor
\[\overline{\sigma}\colon \widetilde{\mathscr{O}}(\Delta^{n+2})^\op \to C_{\times},\]
define $\epsilon(\sigma)\colon \Delta^n \times \Delta^1 \to E$ to be the unique functor which sends the nondegenerate $(n+1)$-simplex
\begin{equation*} (0,0) \to\cdots \to (0,k) \to (1,k) \to \cdots \to (1,n)
\end{equation*}
to the $(n+1)$-simplex $\Delta^{n+1} \to E$ corresponding to the functor
\begin{equation*}
\overline{\sigma} \circ f_{n k}^\op\colon\widetilde{\mathscr{O}}(\Delta^{n+3})^\op \to C_{\times}.
\end{equation*} 
It is easy (albeit tedious) to verify that the functors $\epsilon(\sigma)$ assemble to yield a unique functor $\epsilon$. Now set
\[R\coloneq\epsilon|_{(E \times \{0 \})}.\]
Given an object $\tau \in E$ displayed as a $2$-simplex
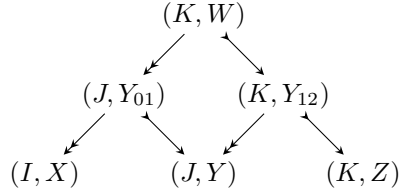
\begin{equation*}
\begin{tikzpicture}[baseline]
\matrix(m)[matrix of math nodes, 
row sep={7ex,between origins}, column sep={7ex,between origins}, 
text height=1.5ex, text depth=0.25ex] 
{&&(K,W)&&\\
&(J,Y_{01})&&(K,Y_{12})&\\ 
(I,X)&&(J,Y)&&(K,Z)\\}; 
\path[>=stealth,->,font=\scriptsize] 
(m-1-3) edge[->>] (m-2-2) 
edge[>->] (m-2-4) 
(m-2-2) edge[->>] (m-3-1)
edge[>->] (m-3-3) 
(m-2-4) edge[->>] (m-3-3)
edge[>->] (m-3-5); 
\end{tikzpicture}
\end{equation*}
of $\Sigma$, the edge $\epsilon_\tau\colon R(\tau) \to \tau$ to be
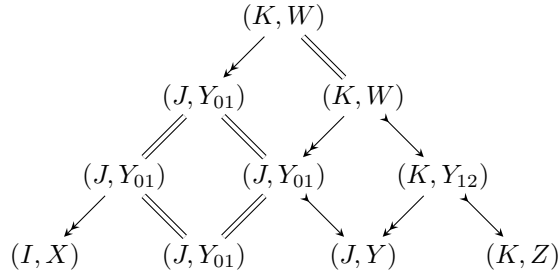
\begin{equation*}
\begin{tikzpicture}[baseline]
\matrix(m)[matrix of math nodes, 
row sep={7ex,between origins}, column sep={7ex,between origins}, 
text height=1.5ex, text depth=0.25ex] 
{&&&(K,W)&&&\\
&&(J,Y_{01})&&(K,W)&&\\ 
&(J,Y_{01})&&(J,Y_{01})&&(K,Y_{12})&\\
(I,X)&&(J,Y_{01})&&(J,Y)&&(K,Z)\\}; 
\path[>=stealth,->,font=\scriptsize] 
(m-1-4) edge[->>] (m-2-3) 
edge[-,double distance=1.5pt] (m-2-5) 
(m-2-3) edge[-,double distance=1.5pt] (m-3-2)
edge[-,double distance=1.5pt] (m-3-4) 
(m-2-5) edge[->>] (m-3-4)
edge[>->] (m-3-6)
(m-3-2) edge[->>] (m-4-1)
edge[-,double distance=1.5pt] (m-4-3)
(m-3-4) edge[-,double distance=1.5pt] (m-4-3)
edge[>->] (m-4-5)
(m-3-6) edge[->>] (m-4-5)
edge[>->] (m-4-7); 
\end{tikzpicture}
\end{equation*}
From this, it is apparent that the essential image $E'$ of $R$ is the full subcategory spanned by those $\tau \in E$ such that the morphism $(J,Y_{01}) \to (J,Y)$ is an equivalence, and by the dual of \cite[Pr. 5.2.7.4]{HTT}, $R$ is a colocalization functor.

We now define $(J, \overline{W}) \in C_\times$ by
\begin{equation*} \overline{W}_j\coloneq\begin{cases}
W_{\beta(j)} & \textup{if } \beta(j) \neq \ast; \\
\varnothing & \textup{if } \beta(j)= \ast.
\end{cases}
\end{equation*}
There is an obvious factorization of $(K,W) \to (I,X)$ through $(J, \overline{W})$, and we define an object $\omega \in E'$ as
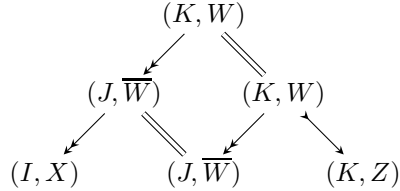
\begin{equation*}
\begin{tikzpicture}[baseline]
\matrix(m)[matrix of math nodes, 
row sep={7ex,between origins}, column sep={7ex,between origins}, 
text height=1.5ex, text depth=0.25ex] 
{&&(K,W)&&\\
&(J,\overline{W})&&(K,W)&\\ 
(I,X)&&(J,\overline{W})&&(K,Z)\\}; 
\path[>=stealth,->,font=\scriptsize] 
(m-1-3) edge[->>] (m-2-2) 
edge[-,double distance=1.5pt] (m-2-4) 
(m-2-2) edge[->>] (m-3-1)
edge[-,double distance=1.5pt] (m-3-3) 
(m-2-4) edge[->>] (m-3-3)
edge[>->] (m-3-5); 
\end{tikzpicture}
\end{equation*}

We now claim that $\omega$ is terminal in $E'$. Let $\tau \in E'$ be any object displayed as a $2$-simplex
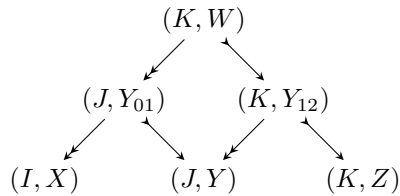
\begin{equation*}
\begin{tikzpicture}[baseline]
\matrix(m)[matrix of math nodes, 
row sep={7ex,between origins}, column sep={7ex,between origins}, 
text height=1.5ex, text depth=0.25ex] 
{&&(K,W)&&\\
&(J,Y_{01})&&(K,Y_{12})&\\ 
(I,X)&&(J,Y)&&(K,Z)\\}; 
\path[>=stealth,->,font=\scriptsize] 
(m-1-3) edge[->>] (m-2-2) 
edge[>->] (m-2-4) 
(m-2-2) edge[->>] (m-3-1)
edge[>->] (m-3-3) 
(m-2-4) edge[->>] (m-3-3)
edge[>->] (m-3-5); 
\end{tikzpicture}
\end{equation*}
of $\Sigma$. We have a homotopy pullback square
\begin{equation*}
\begin{tikzpicture}[baseline]
\matrix(m)[matrix of math nodes,
row sep=6ex, column sep=4ex,
text height=1.5ex, text depth=0.25ex]
{ \Map_E(\tau,\omega) & \Map_{\Sigma_{(I,X)/}}(d_2(\tau),d_2(\omega)) \\
\Delta^0 & \Map_{\Sigma_{(I,X)/}}(d_2(\tau),\widetilde{\gamma}) \\ };
\path[>=stealth,->,font=\scriptsize]
(m-1-1) edge (m-1-2)
edge (m-2-1)
(m-1-2) edge node[right]{$\omega_\ast$} (m-2-2)
(m-2-1) edge node[below]{$\tau$} (m-2-2);
\end{tikzpicture}
\end{equation*}
and the terms on the right-hand side are in turn given as homotopy pullbacks
\begin{equation*}
\begin{tikzpicture}[baseline]
\matrix(m)[matrix of math nodes,
row sep=6ex, column sep=4ex,
text height=1.5ex, text depth=0.25ex]
{ \Map_{\Sigma_{(I,X)/}}(d_2(\tau),d_2(\omega)) & \Map_{\Sigma}((J,Y),(J,\overline{W})) \\
\Delta^0 & \Map_{\Sigma}((I,X),(J,\overline{W})), \\ };
\path[>=stealth,->,font=\scriptsize]
(m-1-1) edge (m-1-2)
edge (m-2-1)
(m-1-2) edge node[right]{$d_2(\tau)^\ast$} (m-2-2)
(m-2-1) edge node[below]{$d_2(\omega)$} (m-2-2);
\end{tikzpicture}
\end{equation*}
and
\begin{equation*}
\begin{tikzpicture}[baseline]
\matrix(m)[matrix of math nodes,
row sep=6ex, column sep=4ex,
text height=1.5ex, text depth=0.25ex]
{ \Map_{\Sigma_{(I,X)/}}(d_2(\tau),\widetilde{\gamma}) & \Map_{\Sigma}((J,Y),(K,Z)) \\
\Delta^0 & \Map_{\Sigma}((I,X),(K,Z)). \\ };
\path[>=stealth,->,font=\scriptsize]
(m-1-1) edge (m-1-2)
edge (m-2-1)
(m-1-2) edge node[right]{$d_2(\tau)^\ast$} (m-2-2)
(m-2-1) edge node[below]{$\widetilde{\gamma}$} (m-2-2);
\end{tikzpicture}
\end{equation*}

In light of the equivalence $\equivto{(J,Y_{01})}{(J,Y)}$, we obtain equivalences
\begin{align*}
\Map_{\Sigma}((J,Y),(J,\overline{W})) & \simeq \prod_{j \in J} \iota C'_{/\{(Y_{01})_j\ ;\ \overline{W}_j\}}; \\
\Map_{\Sigma}((I,X),(J,\overline{W})) & \simeq \prod_{j \in J} \iota C'_{/\{X_i\ ;\ \overline{W}_j\}_{i \in \alpha^{-1}(j)}}.
\end{align*}
Under these equivalences the map $d_2(\tau)^\ast$ is given by $\prod_{j \in J} \phi_j$ where
\begin{equation*}
\phi_j\colon \iota C'_{/\{(Y_{01})_j\ ;\ \overline{W}_j\}} \to \iota C'_{/\{X_i\ ;\ \overline{W}_j\}_{i \in \alpha^{-1}(j)}}
\end{equation*}
is defined by postcomposition by the maps $(Y_{01})_j \to X_i$ (with $i \in \alpha^{-1}(j)$). As a corollary of Cor. \ref{cor:pedantry} below, we may factor the square in question into two homotopy pullback squares:
\begin{equation*}
\begin{tikzpicture}[baseline]
\matrix(m)[matrix of math nodes,
row sep=4ex, column sep=4ex,
text height=1.5ex, text depth=0.25ex]
{ \Map_{\Sigma_{(I,X)/}}(d_2(\tau),d_2(\omega)) & \Map_{(C_{\times})^{\dag}_{\id}}((J,\overline{W}),(J,Y_{01})) & \prod_{j \in J} \iota C'_{/\{(Y_{01})_j\ ;\ \overline{W}_j\}} \\
\Delta^0 & \Map_{(C_{\times})^{\dag}_\alpha}((J,\overline{W}),(I,X)) & \prod_{j \in J} \iota C'_{/\{X_i\ ;\ \overline{W}_j\}_{i \in \alpha^{-1}(j)}}. \\ };
\path[>=stealth,->,font=\scriptsize]
(m-1-1) edge (m-1-2)
edge (m-2-1)
(m-1-2) edge (m-1-3)
edge (m-2-2)
(m-1-3) edge (m-2-3)
(m-2-1) edge (m-2-2)
(m-2-2) edge (m-2-3);
\end{tikzpicture}
\end{equation*}
Similarly, we factor the second square into two homotopy pullback squares:
\begin{equation*}
\begin{tikzpicture}[baseline]
\matrix(m)[matrix of math nodes,
row sep=4ex, column sep=4ex,
text height=1.5ex, text depth=0.25ex]
{ \Map_{\Sigma_{(I,X)/}}(d_2(\tau),\widetilde{\gamma}) & \Map_{(C_{\times})^{\dag}_\beta}((K,W),(J,Y_{01})) & \prod_{k \in K} \iota C'_{/\{(Y_{01})_j\ ;\ Z_k\}_{j \in \beta^{-1}(k)}} \\
\Delta^0 & \Map_{(C_{\times})^{\dag}_\gamma}((K,W),(I,X)) & \prod_{k \in K} \iota C'_{/\{X_i\ ;\ Z_k\}_{i \in \gamma^{-1}(k)}} \\ };
\path[>=stealth,->,font=\scriptsize]
(m-1-1) edge (m-1-2)
edge (m-2-1)
(m-1-2) edge (m-1-3)
edge (m-2-2)
(m-1-3) edge (m-2-3)
(m-2-1) edge (m-2-2)
(m-2-2) edge (m-2-3);
\end{tikzpicture}
\end{equation*}

The map $\omega_\ast$ is then seen to be equivalent to the induced map between the fibers of the horizontal maps in the following commutative square:
\begin{equation*}
\begin{tikzpicture}[baseline]
\matrix(m)[matrix of math nodes,
row sep=4ex, column sep=4ex,
text height=1.5ex, text depth=0.25ex]
{ \Map_{(C_{\times})^{\dag}_{\id}}((J,\overline{W}),(J,Y_{01})) & \Map_{(C_{\times})^{\dag}_{\alpha}}((J,\overline{W}),(I,X)) \\
  \Map_{(C_{\times})^{\dag}_{\beta}}((K,W),(J,Y_{01}))  & \Map_{(C_{\times})^{\dag}_{\gamma}}((K,W),(I,X)). \\ };
\path[>=stealth,->,font=\scriptsize]
(m-1-1) edge (m-1-2)
edge (m-2-1)
(m-1-2) edge (m-2-2)
(m-2-1) edge (m-2-2);
\end{tikzpicture}
\end{equation*}
The left vertical map is the equivalence
\begin{equation*}
\equivto{\prod_{j\in\beta^{-1}(K)} \Map_{C^\dag}(W_{\beta(j)},(Y_{01})_j)}{\prod_{k \in K} \prod_{j \in \beta^{-1}(k)} \Map_{C^\dag}(W_k, (Y_{01})_j),}
\end{equation*}
and the right vertical map is the equivalence
\begin{equation*}
\equivto{\prod_{j\in\beta^{-1}(K)} \prod_{i \in \alpha^{-1}(j)} \Map_{C^\dag}(W_{\beta(j)},X_i)}{\prod_{k \in K} \prod_{i \in \gamma^{-1}(k)} \Map_{C^\dag}(W_k, X_i),}
\end{equation*}
so the square is in fact a homotopy pullback square and $\omega_\ast$ is an equivalence. Hence the mapping space $\Map_E(\tau,\omega)$ is contractible and $\omega$ is a terminal object of $E'$. This proves that $E$ is weakly contractible.
\end{proof}
\end{prp}

We digress briefly to give the following proposition, which is useful for studying the interaction of the over and undercategory functors with homotopy colimit diagrams.

\begin{prp} \label{prp:sliceadj} Suppose $C$ an $\infty$-category, and let $s\Set_{/C}$ be endowed with the model structure created by the forgetful functor to $s\Set$ equipped with the Joyal model structure. Then we have a Quillen adjunction
\[ \adjunct{C_{(-)/}}{s\Set_{/C}}{(s\Set_{/C})^\op}{C_{/(-)}}  \]
between the over and undercategory functors.
\end{prp}
\begin{proof} The displayed functors are indeed adjoint to each other, since for objects $\phi \colon \fromto{X}{C}$ and $\psi \colon \fromto{Y}{C}$ we have natural isomorphisms
\[ \Hom_{/C}(X,C_{\psi /}) \cong \Hom_{(X \sqcup Y)/}(X \star Y, C) \cong \Hom_{/C}(Y, C_{/ \phi}). \]
To check that this adjunction is a Quillen adjunction, we check that $C_{(-)/}$ preserves cofibrations and trivial cofibrations. Let $\tau \colon \fromto{\phi}{\phi'}$ be a map in $s\Set_{/C}$, and let $f = d_2(\tau) \colon \fromto{X}{X'}$. If $f$ is a monomorphism, by \cite[2.1.2.1]{HTT} we have that $C_{\phi'/} \to C_{\phi/}$ is a left fibration, hence by \cite[2.4.6.5]{HTT} a categorical fibration. If $f$ is a monomorphism and a categorical equivalence, by \cite[4.1.1.9]{HTT} and \cite[4.1.1.1(4)]{HTT} $f$ is right anodyne, hence by \cite[2.1.2.5]{HTT} $C_{\phi'/} \to C_{\phi/}$ is a trivial Kan fibration.
\end{proof}

\begin{cor} \label{cor:pedantry} Let $C$ be an $\infty$-category and suppose given a morphism $f: x \to y$ in $C$ and a diagram

\begin{equation*}
\begin{tikzpicture}[baseline]
\matrix(m)[matrix of math nodes,
row sep=4ex, column sep=4ex,
text height=1.5ex, text depth=0.25ex]
{ K  & L & C \\
  K \sqcup \Delta^0 & L \sqcup \Delta^0 \\ };
\path[>=stealth,->,font=\scriptsize]
(m-1-1) edge[right hook->] node[above]{$\phi$} (m-1-2)
edge (m-2-1)
(m-1-2) edge (m-2-2)
edge node[above]{$p$} (m-1-3)
(m-2-1) edge[right hook->] node[above]{$\phi'$} (m-2-2)
(m-2-2) edge node[below]{$p'$} (m-1-3);
\end{tikzpicture}
\end{equation*}
of simplicial sets where $\phi' = \phi \sqcup \id$ and $p'|_{\Delta^0}$ selects $y$. Then we have a homotopy pullback square of $\infty$-categories
\begin{equation*}
\begin{tikzpicture}[baseline]
\matrix(m)[matrix of math nodes,
row sep=4ex, column sep=4ex,
text height=1.5ex, text depth=0.25ex]
{ \{ x \} \times_C C_{/p}  & C_{/p'} \\
  \{ x \} \times_C C_{/p\circ \phi} & C_{/p'\circ \phi'} \\ };
\path[>=stealth,->,font=\scriptsize]
(m-1-1) edge node[above]{$F$} (m-1-2)
edge (m-2-1)
(m-1-2) edge (m-2-2)
(m-2-1) edge node[above]{$G$} (m-2-2);
\end{tikzpicture}
\end{equation*}
where the vertical functors are given by change of diagram and the horizontal functors are to be defined.
\begin{proof}
Define the functor $F$ as follows: the datum of an $n$-simplex $\Delta^n \to \{x \} \times_{C} C_{/p}$ consists of a map $\alpha\colon \Delta^n \star L \to C$ which restricts to $p$ on $L$ and to the constant map to $x$ on $\Delta^n$, and we use this to define $\Delta^n \star (L \sqcup \Delta^0) \to C$ to be the unique map which restricts to $\alpha$ on $\Delta^n \star L$ and to
\[\Delta^n \star \Delta^0 \to \Delta^1 \ \tikz[baseline]\draw[>=stealth,->,font=\scriptsize,inner sep=0.8pt](0,0.5ex)--node[above]{$f$}(0.5,0.5ex);\ C\]
on $\Delta^n \star \Delta^0$; this gives the $n$-simplex of $C_{/p'}$. The definition of $G$ is analogous. The square in question then fits into a rectangle
\begin{equation*}
\begin{tikzpicture}[baseline]
\matrix(m)[matrix of math nodes,
row sep=4ex, column sep=4ex,
text height=1.5ex, text depth=0.25ex]
{ \{ x \} \times_C C_{/p}  & C_{/p'} & C_{/p} \\
  \{ x \} \times_C C_{/p\circ \phi} & C_{/p'\circ \phi'} & C_{/p \circ \phi}\\ };
\path[>=stealth,->,font=\scriptsize]
(m-1-1) edge node[above]{$F$} (m-1-2)
edge (m-2-1)
(m-1-2) edge (m-2-2)
edge (m-1-3)
(m-1-3) edge (m-2-3)
(m-2-1) edge node[above]{$G$} (m-2-2)
(m-2-2) edge (m-2-3);
\end{tikzpicture}
\end{equation*}
where the long horizontal functors are given as the inclusion of the fiber over $x$ and the functors in the righthand square are given by change of diagram. By Prp. \ref{prp:sliceadj} and left properness of the Joyal model structure, the righthand square is a homotopy pullback square. The vertical functor $\fromto{C_{/p}}{C_{/p \circ \phi}}$ is a right fibration, so the outermost square is a homotopy pullback square. The conclusion follows.
\end{proof}
\end{cor}

If we want the symmetric promonoidal $\infty$-category
\[ \fromto{A^{\eff}(C,C_{\dag},C^{\dag})^{\otimes}}{\N\Lambdaup(\FF)} \]
to be symmetric monoidal, we need a nontrivial condition on our disjunctive triple.

\begin{dfn}\label{dfn:cartesian} A disjunctive triple $(C,C_{\dag},C^{\dag})$ will be said to be \textbf{\emph{cartesian}} just in case it enjoys the following properties
\begin{enumerate}[(\ref{dfn:cartesian}.1)]
\item It is left complete.
\item The underlying $\infty$-category $C$ admits finite products.
\item For any object $X\in C$, the product functor
\begin{equation*}
X\times -\colon\fromto{C}{C}
\end{equation*}
preserves finite coproducts; that is, for any finite set $I$ and any collection $\{U_i\ |\ i\in I\}$ of objects of $C$, the natural map
\begin{equation*}
\fromto{\coprod_{i\in I}(X\times U_i)}{X\times\left(\coprod_{i\in I}U_i\right)}
\end{equation*}
is an equivalence.
\item A morphism $\fromto{X}{\prod_{j\in J}Y_j}$ is egressive just in case each of the components $\fromto{X}{Y_j}$ is so.
\end{enumerate}
\end{dfn}

\begin{exm} Note that a disjunctive $\infty$-category $C$ that admits a teminal object, when equipped with the maximal triple structure (in which every morphism is both ingressive and egressive) is always cartesian. More generally, any disjunctive triple that contains a terminal object $1$ with the property that every morphism $\fromto{X}{1}$ is ingressive and egressive is cartesian.
\end{exm}

\begin{prp}\label{prp:cartesiancocart} If $(C,C_{\dag},C^{\dag})$ is a cartesian disjunctive triple, then the symmetric promonoidal $\infty$-category 
\begin{equation*}
p\colon\fromto{A^{\eff}(C,C_{\dag},C^{\dag})^{\otimes}}{\N\Lambdaup(\FF)}
\end{equation*}
is symmetric monoidal; that is, $p$ is a cocartesian fibration.
\begin{proof}  Since $p$ is flat, by Pr. \ref{prp:flatloccocartiscocart} it suffices to verify that $p$ is a locally cocartesian fibration. Since $p$ is an $\infty$-operad, by the dual of \cite[Lm. 2.4.2.7]{HTT} we reduce to checking that for any active edge $\alpha\colon\fromto{I}{J}$ and any object $(I,X)$ over $I$, there exists a locally $p$-cocartesian edge $\widetilde{\alpha}$ covering $\alpha$. For each $j \in J$, let $\widetilde{X}_j = \prod_{i \in \alpha^{-1}(j)} X_i$, and define $\widetilde{\alpha}$ to be 

\begin{equation*}
\begin{tikzpicture}[baseline]
\matrix(m)[matrix of math nodes, 
row sep={7ex,between origins}, column sep={7ex,between origins}, 
text height=1.5ex, text depth=0.25ex] 
{ &(J,\widetilde{X})& \\ 
(I,X) && (J,\widetilde{X}),\\}; 
\path[>=stealth,->,font=\scriptsize] 
(m-1-2) edge[->>] (m-2-1) 
edge[-,double distance=1.5pt] (m-2-3); 
\end{tikzpicture}
\end{equation*}
where the morphism $\fromto{(J,\widetilde{X})}{(I,X)}$ is defined using the various projection maps $\widetilde{X}_{\alpha(i)} \to X_i$. Then $\widetilde{\alpha}$ is a locally $p$-cocartesian edge if for all $(J,Y) \in A^{\eff}(C,C_{\dag},C^{\dag})^{\otimes}_J$, the induced map

\begin{equation*} \widetilde{\alpha}^\ast \colon \fromto{\Map_{A^{\eff}(C,C_{\dag},C^{\dag})^{\otimes}_J}((J,\widetilde{X}),(J,Y))}{\Map_{A^{\eff}(C,C_{\dag},C^{\dag})^{\otimes}_\alpha}((I,X),(J,Y))}
\end{equation*}
is an equivalence. This map is in turn equivalent to the map
\begin{equation*} \prod_{j \in J} \phi_j \colon \fromto{\prod_{j \in J} \iota C'_{/\left\{ \prod_{i \in \alpha^{-1}(j)} X_i\ ;\ Y_j\right\}}} {\prod_{j \in J} \iota C'_{/\{X_i\ ;\ Y_j\}_{i \in \alpha^{-1}(j) }} }
\end{equation*}
where $\phi_j$ is induced by postcomposition by the projection maps $\prod_{i \in \alpha^{-1}(j)} X_i \to X_i$. Since $(C,C_{\dag},C^{\dag})$ is a cartesian disjunctive triple, we have that the functor
\begin{equation*} \fromto{(C^\dag)_{/ \prod_{i \in \alpha^{-1}(j) } X_i} } { (C^\dag)_{/ (X_i, i \in \alpha^{-1}(j))} }
\end{equation*}
is an equivalence. Hence in light of Prp. \ref{prp:sliceadj} we have a homotopy pullback square
\begin{equation*}
\begin{tikzpicture}[baseline]
\matrix(m)[matrix of math nodes,
row sep=4ex, column sep=4ex,
text height=1.5ex, text depth=0.25ex]
{ \prod_{j \in J} \iota C'_{/\left\{ \prod_{i \in \alpha^{-1}(j)} X_i\ ;\ Y_j\right\}}  & \prod_{j \in J} \iota C'_{/\{X_i\ ;\ Y_j\}_{i \in \alpha^{-1}(j) }} \\
  (C^\dag)_{/ \prod_{i \in \alpha^{-1}(j)} X_i}  & (C^\dag)_{/ \{X_i\}_{i \in \alpha^{-1}(j)} } \\ };
\path[>=stealth,->,font=\scriptsize]
(m-1-1) edge node[above]{$\phi_j$} (m-1-2)
edge (m-2-1)
(m-1-2) edge (m-2-2)
(m-2-1) edge (m-2-2);
\end{tikzpicture}
\end{equation*}
where the horizontal maps are equivalences. We deduce that the map $\widetilde{\alpha}^\ast$ is an equivalence, as desired.
\end{proof}
\end{prp}

In light of Lm. \ref{lm:ACCtoNLFiscocart} and Rk. \ref{rem:explicitAeffotimes}, we obtain the following.
\begin{thm} Suppose $(C,C_{\dag},C^{\dag})$ a left complete disjunctive triple. Then the functor
\begin{equation*}
\fromto{A^{\eff}(C,C_{\dag},C^{\dag})^{\otimes}}{\N\Lambdaup(\FF)}
\end{equation*}
exhibits $A^{\eff}(C,C_{\dag},C^{\dag})^{\otimes}$ as a symmetric promonoidal $\infty$-category, the underlying $\infty$-category of which is the effective Burnside $\infty$-category $A^{\eff}(C,C_{\dag},C^{\dag})$. Furthermore, if $(C,C_{\dag},C^{\dag})$ is cartesian, then $A^{\eff}(C,C_{\dag},C^{\dag})^{\otimes}$ is symmetric monoidal.
\end{thm}

\begin{ntn} When $(C,C_{\dag},C^{\dag})$ is a \emph{right} complete disjunctive triple, we may employ duality and write
\begin{equation*}
A^{\eff}(C,C_{\dag},C^{\dag})_{\otimes}\coloneq(A^{\eff}(C,C^{\dag},C_{\dag})^{\otimes})^{\op}.
\end{equation*}
The functor $\fromto{A^{\eff}(C,C_{\dag},C^{\dag})_{\otimes}}{\N\Lambdaup(\FF)^{\op}}$ is then a symmetric promonoidal structure on the Burnside $\infty$-category $A^{\eff}(C,C^{\dag},C_{\dag})^{\op}\simeq A^{\eff}(C,C_{\dag},C^{\dag})$.
\end{ntn}

\begin{nul}\label{nul:easyfactaboutaeffotimes} Suppose $(C,C_{\dag},C^{\dag})$ a cartesian disjunctive triple. Note that the formula
\begin{equation*}
\coprod_{i\in I}(X\times U_i)\simeq X\times\left(\coprod_{i\in I}U_i\right)
\end{equation*}
implies immediately that the tensor product functor
\begin{equation*}
\otimes\colon\fromto{A^{\eff}(C,C_{\dag},C^{\dag})\times A^{\eff}(C,C_{\dag},C^{\dag})}{A^{\eff}(C,C_{\dag},C^{\dag})}
\end{equation*}
preserves direct sums separately in each variable.

More generally, suppose $(C,C_{\dag},C^{\dag})$ a left complete disjunctive triple, suppose $I$ a finite set, and suppose $\{x_i\}_{i\in I}$ a collection of objects of $C$, which we view, by the standard abuse, as an object of $A^{\eff}(C,C_{\dag},C^{\dag})^{\otimes}_I$. Consider the $1$-simplex $\xi_I\colon\fromto{\Deltaup^1}{N\Lambdaup(\FF)}$, and denote by $h^{\{x_i\}_{i\in I}}$ the restriction of the functor 
\[\fromto{A^{\eff}(C,C_{\dag},C^{\dag})^{\otimes}\times_{N\Lambdaup(\FF)}\Deltaup^1}{\Kan}\]
corepresented by $\{x_i\}_{i\in I}$ to $A^{\eff}(C,C_{\dag},C^{\dag})$. Informally, this is the functor
\[\Map_{C^{\otimes}}^{\xi_I}(\{x_i\}_{i\in I},-).\]
Suppose $j\in I$, and suppose $\{\fromto{y_k}{x_j}\}_{k\in K}$ a family of morphisms that together exhibit $x_j$ as the coproduct $\coprod_{k\in K}y_k$. For each $i\in I$ and $k\in K$, write
\[x'_{i,k}\coloneq\begin{cases}
y_k&\text{if }i=j;\\
x_i&\text{if }i\neq j.
\end{cases}\]
Then the natural map
\[\fromto{h^{\{x_i\}_{i\in I}}}{\prod_{k\in K}h^{\{x'_{i,k}\}_{i\in I}}}\]
is an equivalence.
\end{nul}

\begin{nul} For any disjunctive $\infty$-category $C$ that admits a terminal object, the duality functor
\begin{equation*}
D\colon\equivto{A^{\eff}(C)^{\op}}{A^{\eff}(C)}
\end{equation*}
of \cite[Nt. 3.10]{M1} provides duals for the symmetric monoidal $\infty$-category $A^{\eff}(C)^{\otimes}$ \cite[Df. 2.3.5]{MR2555928}. More precisely, for any object $X$ of $A^{\eff}(C)$, there exists an evaluation morphism $\fromto{X\otimes DX}{1}$ given by the diagram
\begin{equation*}
\begin{tikzpicture}[baseline]
\matrix(m)[matrix of math nodes, 
row sep=3ex, column sep=5ex, 
text height=1.5ex, text depth=0.25ex] 
{&[-2ex]X&\\ 
X\times X&&1,\\}; 
\path[>=stealth,->,font=\scriptsize,inner sep=1.5pt] 
(m-1-2) edge node[above left]{$\Deltaup$} (m-2-1) 
edge node[above right]{$!$} (m-2-3); 
\end{tikzpicture}
\end{equation*}
and, dually, there exists a coevaluation morphism $\fromto{1}{DX\otimes X}$ given by the diagram
\begin{equation*}
\begin{tikzpicture}[baseline]
\matrix(m)[matrix of math nodes, 
row sep=3ex, column sep=5ex, 
text height=1.5ex, text depth=0.25ex] 
{&X&[-2ex]\\ 
1&&X\times X.\\}; 
\path[>=stealth,->,font=\scriptsize,inner sep=1.5pt] 
(m-1-2) edge node[above left]{$!$} (m-2-1) 
edge node[above right]{$\Deltaup$} (m-2-3); 
\end{tikzpicture}
\end{equation*}
Since the square
\begin{equation*}
\begin{tikzpicture} 
\matrix(m)[matrix of math nodes, 
row sep=6ex, column sep=8ex, 
text height=1.5ex, text depth=0.25ex] 
{X&X\times X\\ 
X\times X&X\times X\times X\\}; 
\path[>=stealth,->,font=\scriptsize] 
(m-1-1) edge node[above]{$\Deltaup$} (m-1-2) 
edge node[left]{$\Deltaup$} (m-2-1) 
(m-1-2) edge node[right]{$\Deltaup\times\id$} (m-2-2) 
(m-2-1) edge node[below]{$\id\times\Deltaup$} (m-2-2); 
\end{tikzpicture}
\end{equation*}
is a pullback, it follows that the composite
\begin{equation*}
X\to X\otimes DX\otimes X\to X
\end{equation*}
in $A^{\eff}(C)$ is homotopic to the identity. We conclude that $A^{\eff}(C)^{\otimes}$ is a symmetric monoidal $\infty$-category with duals.
\end{nul}

\begin{nul} If $(C,C_{\dag},C^{\dag})$ is a cartesian disjunctive triple, then in general it is not quite the case that the symmetric monoidal $\infty$-category $A^{\eff}(C,C_{\dag},C^{\dag})^{\otimes}$ admits duals. We have an evaluation morphism $\fromto{X\otimes DX}{1}$ in $A^{\eff}(C,C_{\dag},C^{\dag})$ just in case the diagonal $\Deltaup\colon\fromto{X}{X\times X}$ of $C$ is egressive, and the morphism $!\colon\fromto{X}{1}$ is ingressive. We have a coevaluation morphism $\fromto{1}{DX\otimes X}$ in $A^{\eff}(C,C_{\dag},C^{\dag})$ just in case $\Deltaup$ is ingressive and $!$ is egressive.
\end{nul}

\begin{nul} If $(C,C_{\dag},C^{\dag})$ and $(D,D_{\dag},D^{\dag})$ are left complete disjunctive triples, then it is easy to see that a functor of disjunctive triples
\begin{equation*}
f\colon\fromto{(C,C_{\dag},C^{\dag})}{(D,D_{\dag},D^{\dag})}
\end{equation*}
induces a functor of adequate triples
\[\fromto{(C_{\times},(C_{\times})_{\dag},(C_{\times})^{\dag})}{(D_{\times},(D_{\times})_{\dag},(D_{\times})^{\dag})}\]
and thus a morphism of $\infty$-operads
\begin{equation*}
A^{\eff}(f)^{\otimes}\colon\fromto{A^{\eff}(C,C_{\dag},C^{\dag})^{\otimes}}{A^{\eff}(D,D_{\dag},D^{\dag})^{\otimes}}.
\end{equation*}
If, furthermore, $(C,C_{\dag},C^{\dag})$ and $(D,D_{\dag},D^{\dag})$ are cartesian and $f$ preserves finite products, then $A^{\eff}(f)^{\otimes}$ is of course a symmetric monoidal functor.
\end{nul}


\section{Green functors} Andreas Dress \cite{MR0360771} defined Green functors as Mackey functors equipped with certain pairings. Gaunce Lewis \cite{lewis-notes} noticed that these pairings made them commutative monoids for the Day convolution tensor product on the category of Mackey functors. By an old observation of Brian Day \cite[Ex. 3.2.2]{day-thesis}, these are precisely the lax symmetric monoidal additive functors on the effective Burnside category. Thanks to recent work of Saul Glasman \cite{GlasmanDay}, this characterization of monoids for the Day convolution holds in the $\infty$-categorical context as well.

\begin{dfn} We shall say that a symmetric monoidal $\infty$-category $E^{\otimes}$ is \textbf{\emph{additive}} if the underlying $\infty$-category $E$ is additive, and the tensor product functor $\otimes\colon\fromto{E\times E}{E}$ preserves direct sums separately in each variable.
\end{dfn}

\begin{dfn}\label{dfn:Greenfunctor} 
\begin{enumerate}[(\ref{dfn:Greenfunctor}.1)]
\item Suppose $(C,C_{\dag},C^{\dag})$ a left complete disjunctive triple and $E^{\otimes}$ an additive symmetric monoidal $\infty$-category. Then a \textbf{\emph{commutative Green functor}} is a morphism of $\infty$-operads
\begin{equation*}
\fromto{A^{\eff}(C,C_{\dag},C^{\dag})^{\otimes}}{E^{\otimes}}
\end{equation*}
such that the underlying functor $\fromto{A^{\eff}(C,C_{\dag},C^{\dag})}{E}$ preserves direct sums.
\item More generally, if $O^{\otimes}$ is an $\infty$-operad, then an \textbf{\emph{$O^{\otimes}$-Green functor}} is a morphism of $\infty$-operads
\begin{equation*}
\fromto{A^{\eff}(C,C_{\dag},C^{\dag})^{\otimes}\times_{\N\Lambdaup(\FF)}O^{\otimes}}{E^{\otimes}\times_{\N\Lambdaup(\FF)}O^{\otimes}}
\end{equation*}
over $O^{\otimes}$ such that for any object $X$ of the underlying $\infty$-category $O$, the functor
\begin{equation*}
\fromto{A^{\eff}(C,C_{\dag},C^{\dag})\simeq(A^{\eff}(C,C_{\dag},C^{\dag})^{\otimes}\times_{\N\Lambdaup(\FF)}O^{\otimes})_X}{(E^{\otimes}\times_{\N\Lambdaup(\FF)}O^{\otimes})_X\simeq E}
\end{equation*}
preserves direct sums.
\item Similarly, for any perfect operator category $\Phiup$, we may define a \textbf{\emph{$\Phiup$-Green functor}} as a morphism
\begin{equation*}
\fromto{A^{\eff}(C,C_{\dag},C^{\dag})^{\otimes}\times_{\N\Lambdaup(\FF)}\N\Lambdaup(\Phiup)}{E^{\otimes}\times_{\N\Lambdaup(\FF)}\N\Lambdaup(\Phiup)}
\end{equation*}
of $\infty$-operads over $\Phiup$ such that the underlying functor $\fromto{A^{\eff}(C,C_{\dag},C^{\dag})}{E}$ preserves direct sums.
\end{enumerate}
\end{dfn}

\begin{ntn} Suppose $(C,C_{\dag},C^{\dag})$ a left complete disjunctive triple, and suppose $E^{\otimes}$ an additive symmetric monoidal $\infty$-category. For any $\infty$-operad $O^{\otimes}$, let us write, employing the notation of \cite[Df. 2.1.3.1]{HA}
\begin{equation*}
\Green_{O^{\otimes}}(C,C_{\dag},C^{\dag};E^{\otimes})\subset\Alg_{A^{\eff}(C,C_{\dag},C^{\dag})^{\otimes}\times_{\N\Lambdaup(\FF)}O^{\otimes}\ /O^{\otimes}}(E^{\otimes}\times_{\N\Lambdaup(\FF)}O^{\otimes})
\end{equation*}
for the full subcategory spanned by the $O^{\otimes}$-Green functors.
\end{ntn}

\begin{exm} We define \emph{modules} over an \emph{associative Green functor} in this way. Suppose $(C,C_{\dag},C^{\dag})$ a left complete disjunctive triple, and suppose $E^{\otimes}$ an additive symmetric monoidal $\infty$-category. Then we may consider the $\infty$-operad of \cite[Df. 4.2.7]{HA}, which we will denote $\mathrm{LM}^{\otimes}$. The inclusion $\into{\mathrm{Ass}^{\otimes}}{\mathrm{LM}^{\otimes}}$ induces a functor
\begin{equation*}
\fromto{\Green_{\mathrm{LM}^{\otimes}}(C,C_{\dag},C^{\dag};E^{\otimes})}{\Green_{\mathrm{Ass}^{\otimes}}(C,C_{\dag},C^{\dag};E^{\otimes})}.
\end{equation*}
An object $A$ of the target may be called an \textbf{\emph{associative Green functor}}, and an object of the fiber of this functor over $A$ may be called a \textbf{\emph{left $A$-module}}. We write
\begin{equation*}
\Mod^{\ell}_A(C,C_{\dag},C^{\dag};E^{\otimes})\coloneq\Green_{\mathrm{LM}^{\otimes}}(C,C_{\dag},C^{\dag};E^{\otimes})\times_{\Green_{\mathrm{Ass}^{\otimes}}(C,C_{\dag},C^{\dag};E^{\otimes})}\{A\}
\end{equation*}
for the $\infty$-category of left $A$-modules. When $A$ is a commutative Green functor, we will drop the superscript $\ell$.
\end{exm}

The convolution of two Mackey functors will not in general be a Mackey functor, but it can replaced with one by employing a localization (which we might as well call Mackeyification). To prove that convolution followed by Mackeyification defines a symmetric monoidal structure on the $\infty$-category of Mackey functors, it is necessary to show that Mackeyification is \emph{compatible} with the convolution symmetric monoidal structure  in the sense of Lurie \cite[Df. 2.2.1.6, Ex. 2.2.1.7]{HA}.

The following is immediate from \cite[Pr. 6.5]{M1}.
\begin{lem} Suppose $(C,C_{\dag},C^{\dag})$ a disjunctive triple, and suppose $E$ a presentable additive $\infty$-category. Then the $\infty$-category $\Mack(C,C_{\dag},C^{\dag};E)$ is an accessible localization of the $\infty$-category $\Fun(A^{\eff}(C,C_{\dag},C^{\dag}),E)$.
\end{lem}

\begin{ntn} Suppose $(C,C_{\dag},C^{\dag})$ a disjunctive $\infty$-category, and suppose $E$ a presentable additive $\infty$-category. Then write $M$ for the left adjoint to the fully faithful inclusion
\[\into{\Mack(C,C_{\dag},C^{\dag};E)}{\Fun(A^{\eff}(C,C_{\dag},C^{\dag}),E)}.\]
\end{ntn}

\begin{lem} Suppose $(C,C_{\dag},C^{\dag})$ a left complete disjunctive $\infty$-category, and suppose $E^{\otimes}$ a presentable symmetric monoidal additive $\infty$-category. Then the left adjoint $M$ constructed above is compatible in the sense of \cite[Df. 2.2.1.6]{HA} with Glasman's Day convolution symmetric monoidal structure on $\Fun(A^{\eff}(C,C_{\dag},C^{\dag}),E)$.
\begin{proof} For any collection of objects $\{s_i\ |\ i\in I\}$ of $C$, let
\[h^{\{s_i\}}\colon\fromto{A^{\eff}(C,C_{\dag},C^{\dag})}{\Kan}\]
be as in \ref{nul:easyfactaboutaeffotimes}, and for any object $x\in E$, let 
\[-\otimes x\colon\fromto{\Fun(A^{\eff}(C,C_{\dag},C^{\dag}),\Kan)}{\Fun(A^{\eff}(C,C_{\dag},C^{\dag}),E)}\]
be the composition with the tensor product $-\otimes x\colon\fromto{\Kan}{E}$ with spaces \cite[\S 4.]{HTT}. Thus objects of the form $h^{\{s_i\}}\otimes x$ generate the $\infty$-category $\Fun(A^{\eff}(C,C_{\dag},C^{\dag}),E)$ under colimits. It is easy to see that for any functors $f,g\colon\fromto{A^{\eff}(C,C_{\dag},C^{\dag})}{\Kan}$ and any object $x\in E$, the map
\[\fromto{(f\times g)\otimes x}{(f\otimes x)\oplus(g\otimes x)}\]
is an $M$-equivalence; furthermore, the class of $M$-equivalences is the strongly saturated class generated by the canonical morphisms 
\begin{equation*}
\fromto{h^{s\oplus t}\otimes x}{(h^s\otimes x)\oplus(h^t\otimes x)}.
\end{equation*}
This tensor product and the Day convolution are compatible in the sense that there are natural equivalences
\begin{equation*}
(h^s\otimes x)\otimes(h^t\otimes y)\simeq h^{{\{s,t\}}}\otimes(x\otimes y),
\end{equation*}
whence one obtains natural $M$-equivalences
\begin{eqnarray*}
((h^s\otimes x)\oplus(h^t\otimes x))\otimes(h^u\otimes y)&\simeq&((h^s\otimes x)\otimes(h^u\otimes y))\oplus((h^t\otimes x)\otimes(h^u\otimes y))\\
&\simeq&(h^{\{s,u\}}\otimes{x\otimes y})\oplus(h^{\{t,u\}}\otimes{x\otimes y})\\
&\to&(h^{\{s,u\}}\times h^{\{t,u\}})\otimes x\otimes y\\
&\simeq&h^{\{s\oplus t,u\}}\otimes{x\otimes y}\\
&\simeq&h^{s\oplus t}\otimes x\otimes h^u\otimes y.
\end{eqnarray*}
It follows that for any $M$-equivalence $\fromto{X}{Y}$ and any object $Z$ of the $\infty$-category $\Fun(A^{\eff}(C,C_{\dag},C^{\dag}),E)$, the morphism
\begin{equation*}
\fromto{X\otimes Z}{Y\otimes Z}
\end{equation*}
is an $M$-equivalence.
\end{proof}
\end{lem}

\begin{nul} In particular, if $(C,C_{\dag},C^{\dag})$ is a left complete disjunctive triple, and if $E^{\otimes}$ a presentable symmetric monoidal additive $\infty$-cate\-gory, we obtain a symmetric monoidal $\infty$-cate\-gory $\Mack(C,C_{\dag},C^{\dag};E)^{\otimes}$, and, in light of \cite{GlasmanDay}, for any $\infty$-operad $O^{\otimes}$, one obtains an equivalence
\begin{equation*}
\Alg_{O^{\otimes}}(\Mack(C,C_{\dag},C^{\dag};E)^{\otimes})\simeq\Green_{O^{\otimes}}(C,C_{\dag},C^{\dag};E).
\end{equation*}
\end{nul}


\section{Green stabilization} Now let us address the issue of multiplicative structures on the Mackey stabilization, as constructed in \cite[\S 7]{M1}. In particular, we aim to show that if $E$ is an $\infty$-topos, then the Mackey stabilization of a morphism of operads
\[\fromto{A^{\eff}(C,C_{\dag},C^{\dag})^{\otimes}}{E^{\times}}\]
naturally admits the structure of a Green functor
\begin{equation*}
\fromto{A^{\eff}(C,C_{\dag},C^{\dag})^{\otimes}}{\Sp(E)^{\wedge}}.
\end{equation*}

\begin{dfn} Suppose $(C,C_{\dag},C^{\dag})$ a cartesian disjunctive triple, suppose $E$ an $\infty$-topos, and suppose
\begin{equation*}
f\colon\fromto{A^{\eff}(C,C_{\dag},C^{\dag})^{\otimes}}{E^{\times}}\textrm{\quad and\quad}F\colon\fromto{A^{\eff}(C,C_{\dag},C^{\dag})^{\otimes}}{\Sp(E)^{\otimes}}
\end{equation*}
morphisms of $\infty$-operads. Then a morphism of $A^{\eff}(C,C_{\dag},C^{\dag})^{\otimes}$-algebras
\begin{equation*}
\eta\colon\fromto{f}{\Omega^{\infty}\circ F}
\end{equation*}
will be said to exhibit $F$ as the \textbf{\emph{Green stabilization}} of $f$ if $F$ is a Green functor, and if, for any Green functor $R\colon\fromto{A^{\eff}(C,C_{\dag},C^{\dag})^{\otimes}}{\Sp(E)^{\otimes}}$, the map
\begin{equation*}
\fromto{\Map_{\Green_{E_{\infty}}(C,C_{\dag},C^{\dag};\Sp(E)^{\otimes})}(F,R)}{\Map_{\Alg_{A^{\eff}(C,C_{\dag},C^{\dag})^{\otimes}}(E^{\times})}(f,\Omega^{\infty}\circ R)}
\end{equation*}
induced by $\eta$ is an equivalence.
\end{dfn}

The following result is essentially the same as \cite[Pr. 2.1]{K3}.

\begin{prp}\label{prp:DACotimes} Suppose $(C,C_{\dag},C^{\dag})$ a cartesian disjunctive triple. There exists a symmetric monoidal $\infty$-category $\mathrm{D}A(C,C_{\dag},C^{\dag})^{\otimes}$ and a fully faithful symmetric monoidal functor
\begin{equation*}
j^{\otimes}\colon\into{A^{\eff}(C,C_{\dag},C^{\dag})^{\otimes}}{\mathrm{D}A(C,C_{\dag},C^{\dag})^{\otimes}}
\end{equation*}
with the following properties.
\begin{enumerate}[(\ref{prp:DACotimes}.1)]
\item The $\infty$-category $\mathrm{D}A(C,C_{\dag},C^{\dag})$ underlies $\mathrm{D}A(C,C_{\dag},C^{\dag})^{\otimes}$, and the underlying functor of $j^{\otimes}$ is the inclusion
\begin{equation*}
j\colon\into{A^{\eff}(C,C_{\dag},C^{\dag})}{\mathrm{D}A(C,C_{\dag},C^{\dag})}
\end{equation*}
of \cite[Nt. 7.2]{M1}.
\item For any symmetric monoidal $\infty$-category $E^{\otimes}$ whose underlying $\infty$-category admits all sifted colimits such that the tensor product preserves sifted colimits separately in each variable, the induced functor
\begin{equation*}
\fromto{\Alg_{\mathrm{D}A(C,C_{\dag},C^{\dag})^{\otimes}}(E^{\otimes})}{\Alg_{A^{\eff}(C,C_{\dag},C^{\dag})^{\otimes}}(E^{\otimes})}
\end{equation*}
exhibits an equivalence from the full subcategory spanned by those morphisms of $\infty$-operads $A$ whose underlying functor $A\colon\fromto{\mathrm{D}A(C,C_{\dag},C^{\dag})}{E}$ preserves sifted colimits to the full subcategory spanned by those morphisms of $\infty$-operads $B$ whose underlying functor $B\colon\fromto{A^{\eff}(C,C_{\dag},C^{\dag})}{E}$ preserves filtered colimits.
\item The tensor product functor
\begin{equation*}
\otimes\colon\fromto{\mathrm{D}A(C,C_{\dag},C^{\dag})\times\mathrm{D}A(C,C_{\dag},C^{\dag})}{\mathrm{D}A(C,C_{\dag},C^{\dag})}
\end{equation*}
preserves all colimits separately in each variable.
\end{enumerate}
\begin{proof} The only part that is not a consequence of \cite[Pr. 4.8.1.10 and Var. 4.8.1.11]{HA} is the assertion that the tensor product functor
\begin{equation*}
\otimes\colon\fromto{\mathrm{D}A(C,C_{\dag},C^{\dag})\times\mathrm{D}A(C,C_{\dag},C^{\dag})}{\mathrm{D}A(C,C_{\dag},C^{\dag})}
\end{equation*}
preserves direct sums separately in each variable. This assertion holds for objects of the effective Burnside category $A^{\eff}(C,C_{\dag},C^{\dag})$ thanks to the universality of coproducts in $C$; the general case follows by exhibiting any object of $\mathrm{D}A(C,C_{\dag},C^{\dag})$ as a colimit of a sifted diagram of objects of $A^{\eff}(C,C_{\dag},C^{\dag})$ and using the fact that both the tensor product and the direct sum commute with sifted colimits.
\end{proof}
\end{prp}

In light of \cite[Pr. 3.5]{K3} and \cite[Pr. 6.2.4.14 and Th. 6.2.6.2]{HA}, we now have the following.

\begin{prp} Suppose $(C,C_{\dag},C^{\dag})$ a disjunctive triple, suppose $E$ an $\infty$-topos, and suppose
\begin{equation*}
f\colon\fromto{A^{\eff}(C,C_{\dag},C^{\dag})^{\otimes}}{E^{\times}}
\end{equation*}
a morphism of $\infty$-operads. Then a Green stabilization of $f$ exists. In particular, the functor
\begin{equation*}
\Omegaup^{\infty}\circ-\colon\fromto{\Green(C,C_{\dag},C^{\dag};\Sp(E)^{\otimes})}{\Alg_{A^{\eff}(C,C_{\dag},C^{\dag})^{\otimes}}(E^{\times})}
\end{equation*}
admits a left adjoint that covers the left adjoint of the functor
\begin{equation*}
\Omegaup^{\infty}\circ-\colon\fromto{\Mack(C,C_{\dag},C^{\dag};\Sp(E))}{\Fun(A^{\eff}(C,C_{\dag},C^{\dag}),E)}.
\end{equation*}
\end{prp}

\begin{exm} Suppose $(C,C_{\dag},C^{\dag})$ a cartesian disjunctive triple. Then the functor
\[\fromto{A^{\eff}(C,C_{\dag},C^{\dag})}{\Kan}\]
corepresented by the terminal object $1$ of $C$ is the unit for the Day convolution symmetric monoidal structure of Glasman, and hence it is an $E_{\infty}$ algebra in an essentially unique fashion. Thus we can consider its Green stabilization
\begin{equation*}
\SS^{\otimes}=\SS_{(C,C_{\dag},C^{\dag})}^{\otimes}\colon\fromto{A^{\eff}(C,C_{\dag},C^{\dag})^{\otimes}}{\Sp^{\wedge}},
\end{equation*}
whose underlying Mackey functor is the Burnside Mackey functor $\SS_{(C,C_{\dag},C^{\dag})}$ of \cite{M1}. We call $\SS^{\otimes}$ the \textbf{\emph{Burnside Green functor}}.
\end{exm}

In a similar vein, we immediately have the following:
\begin{prp} For any cartesian disjunctive triple $(C,C_{\dag},C^{\dag})$, the functor
\[\fromto{A^{\eff}(C,C_{\dag},C^{\dag})^\op}{\Mack(C,C_{\dag},C^{\dag};\Sp)}\]
given by the assignment $\goesto{X}{\SS^X}$ is naturally symmetric monoidal. That is, for any two objects $X,Y\in C$, one has a canonical equivalence
\[\SS^X\otimes\SS^Y\simeq\SS^{X \times Y}\]
\end{prp}

\begin{cor} Suppose $(C,C_{\dag},C^{\dag})$ a cartesian disjunctive triple. For any spectral Mackey functor $M$ thereon, write $F(M,-)$ for the right adjoint to the functor
\[-\otimes M\colon\fromto{\Mack(C,C_{\dag},C^{\dag};\Sp)}{\Mack(C,C_{\dag},C^{\dag};\Sp)}.\]
Then for any object $X\in C$, the Mackey functor $F(\SS^X,M)$ is given by the assignment
\[\goesto{Y}{M(X \times Y)}.\]
\end{cor}

The following is now immediate.
\begin{prp} Suppose $(C,C_{\dag},C^{\dag})$ a cartesian disjunctive triple. The Burnside Mackey functor $\SS_{(C,C_{\dag},C^{\dag})}$ is the unit in the symmetric monoidal $\infty$-category $\Mack(C,C_{\dag},C^{\dag};\Sp)^{\otimes}$. Consequently, the Burnside Green functor $\SS_{(C,C_{\dag},C^{\dag})}^{\otimes}$ is the initial object in the $\infty$-category $\Green_{\N\Lambdaup(\FF)}(C,C_{\dag},C^{\dag};\Sp^{\otimes})$, and the forgetful functor
\begin{equation*}
\equivto{\Mod_{\SS^{\otimes}}(C,C_{\dag},C^{\dag};\Sp^{\otimes})}{\Mack(C,C_{\dag},C^{\dag};\Sp)}
\end{equation*}
is an equivalence.
\end{prp}


\section{Duality} In this section, suppose $C$ a disjunctive $\infty$-category that admits a terminal object. Since the functor $\goesto{X}{\SS^X}$ is symmetric monoidal, it follows immediately that every representable Mackey functor $\SS^X$ is strongly dualizable, and
\[(\SS^X)^{\vee}\simeq\SS^{DX}\]

\begin{ntn} For any associative spectral Green functor $R$ and for any object $X\in C$, denote by $R^X$ the left $R$-module $R\otimes\SS^X$, and denote by ${}^X\!R$ the right $R$-module $\SS^X\otimes R$.

Of course for any left (respectively, right) $R$-module $M$, one has
\[\Map(R^X,M)\simeq\Omegaup^\infty M(X)\text{\qquad(resp.,\quad}\Map({}^X\!R,M)\simeq\Omegaup^\infty M(X)\text{\quad)}.\]
\end{ntn}

\begin{dfn} For any associative spectral Green functor $R$ on $C$, denote by $\Perf^{\ell}_R$ the smallest stable subcategory of the $\infty$-category $\Mod^{\ell}_R$ that contains the left $R$-modules $R^X$ (for $X\in C$) and is closed under retracts. Similarly, denote by $\Perf^{r}_R$ the smallest stable subcategory of the $\infty$-category $\Mod^{r}_R$ that contains the right $R$-modules ${}^X\!R$ (for $X\in C$) and is closed under retracts.

The objects of $\Perf^{\ell}_R$ (respectively, $\Perf^{r}_R$) will be called \textbf{\emph{perfect}} left (resp., right) modules over $R$.
\end{dfn}

Now we obtain the following, which is a straightforward analogue of \cite[Pr. 7.2.5.2]{HA}.
\begin{prp} For any associative spectral Green functor $R$, a left $R$-module is compact just in case it is perfect.
\begin{proof} For any $X\in C$, the functor corepresented by $R^X$ is the assignment $\goesto{M}{\Omegaup^\infty M(X)}$, which preserves filtered colimits. Hence $R^X$ is compact, and thus any perfect left $R$-module is compact.

Conversely, there is a fully faithful, colimit-preserving functor
\[F\colon\into{\Ind(\Perf^{\ell}_R)}{\Mod_R}\]
induced by the inclusion $\into{\Perf^{\ell}_R}{\Mod^{\ell}_R}$. If this is not essentially surjective, there exists a nonzero left $R$-module $M$ such that for every $R$-module $N$ in the essential image of $F$, the group $[N,M]$ vanishes. In particular, for any integer $n$ and any object $X\in C$,
\[\pi_nM(X)\cong[R^X[n],M]\cong 0,\]
whence $M\simeq 0$.
\end{proof}
\end{prp}

The proof of the following is word-for-word identical to that of \cite[Pr. 7.2.5.4]{HA}.
\begin{prp} For any associative spectral Green functor $R$ on $C$, a left $R$-module $M$ is perfect just in case there exists a right $R$-module $M^{\vee}$ that is dual to $M$ in the sense that the functor
\[\Map(\SS,M^{\vee}\otimes_R-)\colon\fromto{\Mod^{\ell}_R}{\Kan}\]
is the functor that $M$ corepresents.
\end{prp}

\begin{exm} Note that, in particular, for any object $X\in C$, one has
\[(R^X)^{\vee}\simeq{}^{DX}\!R.\]
\end{exm}


\section{The Künneth spectral sequence} Let us note that the Künneth spectral sequence works in the Mackey functor context more or less exactly as in the ordinary $\infty$-category of spectra. To this end, let us first discuss $t$-structures on $\infty$-categories of spectral Mackey functors.

\begin{prp} Suppose $(C,C_\dag,C^\dag)$ a disjunctive triple, and suppose $A$ a stable $\infty$-category equipped with a $t$-structure $(A_{\geq0},A_{\leq 0})$. Then the two subcategories
\[\Mack(C,C_\dag,C^\dag;A)_{\geq 0}\coloneq\Mack(C,C_\dag,C^\dag;A_{\geq 0})\]
and
\[\Mack(C,C_\dag,C^\dag;A)_{\leq 0}\coloneq\Mack(C,C_\dag,C^\dag;A_{\leq 0})\]
define a $t$-structure on $\Mack(C,C_\dag,C^\dag;A)$.
\begin{proof} Consider the functor $L\colon\fromto{\Mack(C,C_\dag,C^\dag;A)}{\Mack(C,C_\dag,C^\dag;A)}$ given by composition with $\tau_{\leq -1}$; it is clear that $L$ is a localization functor. Furthermore, the essential image of $L$ is the $\infty$-category $\Mack(C,C_\dag,C^\dag;A_{\leq -1})$, which is closed under extensions, since $A_{\leq -1}$ is. Now we apply \cite[Pr. 1.2.1.16]{HA}.
\end{proof}
\end{prp}

\begin{nul} Note that if $A$ a stable $\infty$-category equipped with a $t$-structure $(A_{\geq0},A_{\leq 0})$, then for any disjunctive triple $(C,C_\dag,C^\dag)$, the heart of the induced $t$-structure on $\Mack(C,C_\dag,C^\dag;A)$ is given by
\[\Mack(C,C_\dag,C^\dag;A)^{\heartsuit}\simeq\Mack(C,C_\dag,C^\dag;A^{\heartsuit}).\]

Furthermore, it is clear that many properties of the $t$-structure on $A$ are inherited by the induced $t$-structure $\Mack(C,C_\dag,C^\dag;A)$: in particular, one verifies easily that the $t$-structure on $\Mack(C,C_\dag,C^\dag;A)$ is left bounded, right bounded, left complete, right complete, compatible with sequential colimits, compatible with filtered colimits, or accessible if the $t$-structure on $A$ is so.
\end{nul}

\begin{exm} For any disjunctive triple $(C,C_\dag,C^\dag)$, the $\infty$-category of spectral Mackey functors $\Mack(C,C_\dag,C^\dag;\Sp)$ admits an accessible $t$-structure that is both left and right complete whose heart is the abelian category $\Mack(C,C_\dag,C^\dag;N\Ab)$. Observe that the corepresentable functors $\tau_{\leq 0} \SS^X$ are projective objects in the heart, and thus the heart has enough projectives.

In particular, if $G$ is a profinite group and if $C$ is the disjunctive $\infty$-category of finite $G$-sets, then the $\infty$-category $\Mack_G$ of spectral Mackey functors for $G$ admits an accessible $t$-structure that is both left and right complete, in which the heart $\Mack_G^{\heartsuit}$ is the nerve of the usual abelian category of Mackey functors for $G$.
\end{exm}

\begin{cnstr}\label{sseqcnstr} Suppose $A$ a stable $\infty$-category equipped with a $t$-structure. Suppose $(C,C_\dag,C^\dag)$ a disjunctive triple, and suppose $X\colon\fromto{N\ZZ}{\Mack(C,C_\dag,C^\dag;A)}$ a filtered Mackey functor with colimit $X(+\infty)$. Then we have the spectral sequence
\[E^{p,q}_r\coloneq\im\left[\fromto{\pi_{p+q}\left(\frac{X(p)}{X(p-r)}\right)}{\pi_{p+q}\left(\frac{X(p+r-1)}{X(p-1)}\right)}\right]\]
associated with $X$ \cite[Df. 1.2.2.9]{HA}.

Note that this is a spectral sequence of $A^{\heartsuit}$-valued Mackey functors. Since limits and colimits of Mackey functors are defined objectwise, it follows that for any object $U\in A^{\eff}(C,C_{\dag},C^{\dag})$, the value $E^{p,q}_r(U)$ is the spectral sequence (in $A^{\heartsuit}$) associated with the filtered object $X(U)\colon\fromto{N\ZZ}{A}$.
\end{cnstr}

\begin{nul} In the setting of Cnstr. \ref{sseqcnstr}, assume that $A$ admits all sequential colimits and that the $t$-structure is compatible with these colimits. If $X(n)\simeq 0$ for $n\ll 0$, then the associated spectral sequence converges to a filtration on $\pi_{p+q}(X(+\infty))$ \cite[1.2.2.14]{HA}. That is:
\begin{itemize}
\item For any $p$ and $q$, there exists $r\gg 0$ such that the differential
\[d_r\colon\fromto{E^{p,q}_r}{E^{p-r,q+r-1}_r}\]
vanishes.
\item For any $p$ and $q$, there exist a discrete, exhaustive filtration
\[\cdots\subset F_{p+q}^{-1}\subset F_{p+q}^0\subset F_{p+q}^1\subset\cdots\subset\pi_{p+q}X(+\infty)\]
and an isomorphism $E^{p,q}_\infty\cong F_{p+q}^p/F_{p+q}^{p-1}$.
\end{itemize}

In more general circumstances, one can obtain a kind of ``local convergence.'' Suppose again that $A$ admits all sequential colimits, and that the $t$-structure is compatible with these colimits. Now suppose that for every object $U\in A^{\eff}(C,C_{\dag},C^{\dag})$, there exists $n\ll 0$ such that $X(n)(U)\simeq 0$. Then for every object $U\in A^{\eff}(C,C_{\dag},C^{\dag})$, the spectral sequence $E^{p,q}_r(U)$ converges to $\pi_{p+q}(X(+\infty)(U))$. In finitary cases (e.g., when $C$ is the disjunctive $\infty$-category of finite $G$-sets for a finite group $G$), there is no difference between the local convergence and the global convergence.

Better convergence results can be obtained when the filtered Mackey functor is the skeletal filtration of a simplicial connective object $Y_\ast$ \cite[Pr. 1.2.4.5]{HA}. In this case, we do not need to assume that the $t$-structure on $A$ is compatible with sequential colimits, the associated spectral sequence is a first-quadrant spectral sequence, and it converges to a length $p+q$ filtration on $\pi_{p+q}|Y_\ast|$.
\end{nul}

Now, to construct the Künneth spectral sequence for Mackey functors, we can follow very closely the arguments of Lurie \cite[\S 7.2.1]{HA}.

\begin{lem} \label{lem:projgen} Suppose $(C,C_{\dag},C^{\dag})$ a disjunctive triple. Then the collection of corepresentable Mackey functors $\{ \SS^X\ |\ X \in A^{\eff}(C,C_{\dag},C^{\dag}) \}$ is a set of compact projective generators for $\Mack(C,C_{\dag},C^{\dag};\Sp_{\geq 0})$ in the sense of \cite[Dfn. 5.5.2.3]{HTT}.
\end{lem}
\begin{proof} The corepresentable functors provide a set of compact projective generators for the $\infty$-category $\Fun^\times(A^{\eff}(C,C_{\dag},C^{\dag}),\Kan)$ because this category is precisely $P_\Sigma(A^{\eff}(C,C_{\dag},C^{\dag})^\op)$. The functor
\[ \fromto{\Omega^\infty \circ - \colon \Mack(C,C_{\dag},C^{\dag};\Sp_{\geq 0})}{\Fun^\times(A^{\eff}(C,C_{\dag},C^{\dag}),\Kan)} \]
preserves sifted colimits and is conservative, since $\Omega^\infty \colon \Sp_{\geq 0} \to \Kan$ preserves sifted colimits by \cite[1.4.3.9]{HA} and is conservative, and the inclusion of both sides into all functors preserves sifted colimits (we use that $\Kan$ is cartesian closed). We conclude by applying \cite[4.7.4.18]{HA}.
\end{proof}

To set up the spectral sequence we need to impose the hypotheses of strong dualizability on the $\SS^X$. Because of this, we now work in the generality of $C$ a disjunctive $\infty$-category which admits a terminal object.

Suppose
\[R\colon\fromto{A^{\eff}(C)^{\otimes}\times_{N\Lambdaup(\FF)}\mathrm{Ass}^{\otimes}}{\Sp^{\wedge}\times_{N\Lambdaup(\FF)}\mathrm{Ass}^{\otimes}}\]
an associative Green functor, suppose $M$ a right $R$-module, and suppose $N$ a left $R$-module. There is a comparison map
\[ \fromto{\Tor^{\pi_\ast R}_0(\pi_{\ast} M ,\pi_{\ast} N) }{\pi_{\ast} (M \otimes_R N)} \]
constructed as follows: given $x \in \pi_m M (U)$ and $y \in \pi_n N (V)$, choose representatives $\Sigmaup^m ({}^U\!R) \to M$ and $\Sigmaup^n(R^V) \to N$ and take their smash product to obtain a map
\[ \Sigmaup^{m+n} (\SS^{U \times V}) \to \Sigmaup^{m+n} (\SS^{U \times V}) \otimes R \simeq \Sigmaup^m ({}^U\!R) \otimes_R \Sigmaup^n (R^V) \to M \otimes_R N \]
and thus an element $x \otimes y \in \pi_{m+n}(M \otimes_R N)(U \times V)$; this is suitably natural so that it descends to a map out of the Day convolution tensor product $\pi_{\ast} M \otimes_{\pi_\ast R} \pi_{\ast} N$ to $\pi_{\ast} (M \otimes_R N)$. This map is not usually an isomorphism. Instead, we construct a spectral sequence that converges to $\pi_{\ast}(M\otimes_R N)$, in which this map appears as an edge homomorphism.

Let $S$ denote the class of left $R$-modules of the form $\Sigmaup^n R^X$ for $n\in\ZZ$ and $X \in C$. By \cite[Pr. 7.2.1.4]{HA}, there exists an $S$-free $S$-hypercovering $\fromto{P_\bullet}{N}$ in the (presentable) stable $\infty$-category $\Mod^{\ell}_R$.

\begin{lem} For any $S$-hypercovering $\fromto{P_\bullet}{N}$, we have that $\left| P_\bullet \right| \simeq N$.
\end{lem} 
\begin{proof} Let $S_{\geq n}$ be the subset of $S$ on $\Sigmaup^m\circ R^X$ for $m \geq n$. From our $S$-hypercovering $\fromto{P_\bullet}{N}$, we obtain $S_{\geq n}$-hypercoverings $\fromto{\tau_{\geq n} P_\bullet}{\tau_{\geq n} N}$ for every $n \in \ZZ$. Since the $\Sigma^n S^X$, $X \in C$ constitute a set of projective generators for $\Mack(C;\Sp_{\geq n})$ by Lm. \ref{lem:projgen}, we have that $\left| \tau_{\geq n} P_\bullet \right| \simeq \tau_{\geq n} N$ by the hypercompleteness of $\Kan$. By the right completeness of the $t$-structure, we deduce that $\left| P_\bullet \right| \simeq N$.
\end{proof}

By passing to the skeletal filtration of $M\otimes_R|P_{\bullet}|$, we obtain a spectral sequence $\{E_r^{p,q},d_r\}_{r\geq 1}$ that converges to $\pi_{p+q}(M\otimes_RN)$. The complex $(E_1^{\ast,q},d_1)$ is the normalized chain complex $N_{\ast}(\pi_q(M\otimes_RP_{\bullet}))$. 


To proceed, we need to prove the following analogue of \cite[Pr. 7.2.1.17]{HA}.

\begin{lem} If $P$ is a direct sum of objects in $S$, then the map
\[ \fromto{\Tor_0^{\pi_{\ast}R}(\pi_{\ast}M,\pi_{\ast}P)}{\pi_\ast(M\otimes_R P)} \]
is an isomorphism.
\end{lem}
\begin{proof} Both sides commute with direct sums and shifts, so we reduce to the case of $P = R^X$.  We claim first that for any spectral Mackey functor E,
\[ \pi_\ast E \otimes \tau_{\leq 0} \SS^X \cong \pi_{\ast} (E \otimes \SS^X).\]
Since $\tau_{\leq 0} \SS^Y$ corepresents evaluation at $Y$ for $\Ab$-valued Mackey functors, and $\tau_{\leq 0} \SS^X$ has dual $\tau_{\leq 0} \SS^{DX}$, we have $(\pi_\ast E \otimes \tau_{\leq 0} \SS^X)(Y) \cong (\pi_\ast E)(Y \times DX)$. Similarly, corepresentability and strong dualizability on the level of the $\Sp$-valued Mackey functors implies that $\pi_{\ast} (E \otimes \SS^X) (Y) \cong (\pi_{\ast} E) (Y \times DX)$, so we conclude. Now we apply this claim both for $M$ and $R$ to see that
\begin{align*} \pi_\ast M \otimes_{\pi_\ast R} \pi_{\ast} (R^X) & \cong \pi_\ast M \otimes_{\pi_\ast R} (\pi_\ast R \otimes \tau_{\leq 0} \SS^X) \\
	& \cong \pi_\ast M \otimes \tau_{\leq 0} \SS^X \\
	& \cong \pi_{\ast}(M \otimes \SS^X) \\
	& \cong \pi_{\ast}(M \otimes_R R^X).
\end{align*}
We leave the identification of the specified map with this isomorphism to the reader.
\end{proof}

We thus obtain an isomorphism
\[\Tor_0^{\pi_{\ast}R}(\pi_{\ast}M,\pi_{\ast}P_{\bullet})\cong\pi_\ast(M\otimes_RP_{\bullet}).\]

As $P_{\bullet}$ is an $S$-free $S$-hypercovering of $N$, $N_{\ast}(\pi_{\ast}P_{\bullet})$ is a resolution of $\pi_{\ast}N$ by projective $\pi_{\ast}R$-modules. It follows that the $E_2$ page is given by
\[E_2^{p,q}\cong\Tor_p^{\pi_{\ast}R}(\pi_{\ast}M,\pi_{\ast}N)_q.\]

As in \cite[Cor. 7.2.1.23]{HA}, we have an immediate corollary.
\begin{cor} Suppose $C$, $R$, $M$, and $N$ as above. Suppose that $R$, $M$, and $N$ are all connective. Then $M\otimes_RN$ is connective, and one has an isomorphism of ordinary Mackey functors
\[\pi_0(M\otimes_RN)\cong\pi_0M\otimes_{\pi_0R}\pi_0N.\]
\end{cor}

\begin{exm} If $C$ is the category of finite $G$-sets for $G$ a finite group, then our Künneth spectral sequence recovers that of Lewis and Mandell in \cite{Lewis01032006}. We refer the reader there to a more extensive discussion of this spectral sequence in that particular case.
\end{exm}


\section{Symmetric monoidal Waldhausen bicartesian fibrations} In \cite{K1}, we define an $O^{\otimes}$-monoidal Waldhausen $\infty$-category for any $\infty$-operad $O^{\otimes}$ as an $O^{\otimes}$-algebra in the symmetric monoidal $\infty$-category $\Wald_{\infty}^{\otimes}$. We give two equivalent fibrational formulations of this notion.

\begin{dfn}\label{dfn:OmonoidalWald} Suppose $O^{\otimes}$ an $\infty$-operad. An \textbf{\emph{$O^{\otimes}$-monoidal Waldhausen $\infty$-category}} consists of a pair cocartesian fibration \cite[Df. 3.8]{K1}
\begin{equation*}
p^{\otimes}\colon\fromto{\XX^{\otimes}}{O^{\otimes}}
\end{equation*}
such that the following conditions obtain.
\begin{enumerate}[(\ref{dfn:OmonoidalWald}.1)]
\item The composite
\begin{equation*}
\XX^{\otimes}\to O^{\otimes}\to \N\Lambdaup(\FF)
\end{equation*}
exhibits $\XX^{\otimes}$ as an $\infty$-operad.
\item The fiber $p\colon\fromto{\XX}{O}$ over $\ast\in \N\Lambdaup(\FF)$ is a Waldhausen cocartesian fibration.
\item For any finite set $I$ and any choice of inert morphisms $\{\rho_i\colon\fromto{s}{s_i}\}_{i\in I}$ covering the inert morphisms $\fromto{I}{\{i\}}$, an edge $\eta$ of $\XX^{\otimes}_{s}$ is ingressive if and only if, for every $i\in I$, the edge $\rho_{i,!}(\eta)$ of $\XX_{s_i}$ is ingressive.
\item For any finite set $I$, any morphism $\mu\colon\fromto{s}{t}$ of $O^{\otimes}$ covering the unique active morphism $\fromto{I}{\{\xi\}}$, and any choice of inert morphisms $\{\fromto{s}{s_i}\ |\ i\in I\}$ covering the inert morphisms $\fromto{I}{\{i\}}$, the functor of pairs
\begin{equation*}
\mu_!\colon\fromto{\prod_{i\in I}\XX_{s_i}\simeq\XX^{\otimes}_{s}}{\XX_{t}}
\end{equation*}
is exact separately in each variable \cite{K3}.
\suspend{enumerate}

Dually, suppose $O_{\otimes}$ an $\infty$-anti-operad. Then a \textbf{\emph{$O_{\otimes}$-monoidal Waldhausen $\infty$-category}} is a pair cartesian fibration
\begin{equation*}
p_{\otimes}\colon\fromto{\XX_{\otimes}}{O_{\otimes}}
\end{equation*}
such that the following conditions obtain.
\resume{enumerate}[{[(\ref{dfn:OmonoidalWald}.1)]}]
\item The composition
\begin{equation*}
\XX_{\otimes}\to O_{\otimes}\to \N\Lambdaup(\FF)^{\op}
\end{equation*}
exhibits $\XX_{\otimes}$ as an $\infty$-anti-operad.
\item The fiber $p\colon\fromto{\XX}{O}$ over $\ast\in \N\Lambdaup(\FF)^{\op}$ is a Waldhausen cartesian fibration.
\item For any finite set $I$ and any choice of inert morphisms $\{\pi_i\colon\fromto{s}{s_i}\}_{i\in I}$ covering the inert morphisms $\fromto{I}{\{i\}}$, an edge $\eta$ of $\XX^{\otimes}_{s}$ is ingressive if and only if, for every $i\in I$, the edge $\pi_i^{\star}(\eta)$ of $\XX_{s_i}$ is ingressive.
\item For any finite set $I$, any morphism $\mu\colon\fromto{t}{s}$ of $O_{\otimes}$ covering the opposite of the unique active morphism $\fromto{I}{\{\xi\}}$, and any choice of inert morphisms $\{\fromto{s_i}{s}\}_{i\in I}$ covering the inert morphisms $\fromto{I}{\{i\}}$, the functor of pairs
\begin{equation*}
\mu^{\star}\colon\fromto{\prod_{i\in I}\XX_{s_i}\simeq\XX_{\otimes,s}}{\XX_{t}}
\end{equation*}
is exact separately in each variable.
\end{enumerate}
\end{dfn}

Employing \cite[Ex. 2.4.2.4 and Pr. 2.4.2.5]{HA} and \cite[Lm 1.4]{K3}, one deduces the following.
\begin{prp}\label{prp:Waldbicartmon} Suppose $O^{\otimes}$ (respectively, $O_{\otimes}$) an $\infty$-operad (resp., an $\infty$-anti-operad). Then the functor
\[\fromto{O^{\otimes}}{\Cat_{\infty}}\textrm{\qquad(resp., the functor\quad}\fromto{(O_{\otimes})^{\op}}{\Cat_{\infty}}\textrm{\quad)}\]
classifying an $O^{\otimes}$-monoidal Waldhausen $\infty$-category (resp., an $O_{\otimes}$-monoidal Waldhausen $\infty$-category) factors through an essentially unique morphism of $\infty$-operads
\[\fromto{O^{\otimes}}{\Wald_{\infty}^{\otimes}}\textrm{\qquad(resp., the functor\quad}\fromto{(O_{\otimes})^{\op}}{\Wald_{\infty}^{\otimes}}\textrm{\quad)}\]
\end{prp}

\begin{dfn}\label{dfn:symmonWaldbicart} Now suppose $(C,C_{\dag},C^{\dag})$ a left complete disjunctive triple. A \textbf{\emph{symmetric monoidal Waldhausen bicartesian fibration}}
\begin{equation*}
p_{\boxtimes}\colon\fromto{\XX_{\boxtimes}}{C_{\times}}
\end{equation*}
over $(C,C_{\dag},C^{\dag})$ is a functor of pairs $\fromto{\XX_{\boxtimes}}{(C_{\times})^{\flat}}$ with the following properties.
\begin{enumerate}[(\ref{dfn:symmonWaldbicart}.1)]
\item The underlying functor $p_{\boxtimes}\colon\fromto{\XX_{\boxtimes}}{C_{\times}}$ is an inner fibration.
\item For any egressive morphism $(\phi,\omega):\fibto{(I,X)}{(J,Y)}$ of $C_{\times}$ (in the sense of Nt. \ref{ntn:triplestronCtimes}) and for any object $Q$ of the fiber $(\XX_{\boxtimes})_{(J,Y)}$, there exists a $p_{\boxtimes}$-cartesian morphism $\fromto{P}{Q}$ covering $(\phi,\omega)$.
\item The composition
\[\XX_{\boxtimes}\to C_{\times}\to \N\Lambdaup(\FF)^{\op}\]
exhibits $\XX_{\boxtimes}$ as an $\infty$-anti-operad.
\item The fiber $p\colon\fromto{\XX}{C}$ over $\ast\in \N\Lambdaup(\FF)^{\op}$ is a Waldhausen bicartesian fibration $\fromto{\XX}{C}$ over $(C,C_{\dag},C^{\dag})$.
\end{enumerate}
\end{dfn}

\begin{nul} This is a lot of data, so let's unpack it a bit.

First, a symmetric monoidal Waldhausen bicartesian fibration
\[p_{\boxtimes}\colon\fromto{\XX_{\boxtimes}}{C_{\times}}\]
over $(C,C_{\dag},C^{\dag})$ admits an underlying Waldhausen bicartesian fibration $p\colon\fromto{\XX}{C}$ over $(C,C_{\dag},C^{\dag})$. This provides, for any object $S\in C$, a Waldhausen $\infty$-category $\XX_S$, and for any morphism $\phi\colon\fromto{S}{T}$ of $C$, it provides an exact ``pushforward'' functor $\phi_!\colon\fromto{\XX_S}{\XX_T}$ whenever $\phi$ is ingressive and an exact ``pullback'' functor $\phi^{\star}\colon\fromto{\XX_T}{\XX_S}$ whenever $\phi$ is egressive. These are compatible with composition, and when $\phi$ is ingressive and (therefore) egressive, these two are adjoint.

There's more structure here: for any finite set $I$ and any $I$-tuple $(S_i)_{i\in I}$ of objects of $C$ with product $S$, consider the cartesian edge
\[\fromto{(\{\xi\},S)}{(I,S_I)}\]
of $C_\times$ lying over the morphism $\fromto{\{\xi\}}{I}$ of $\Lambdaup(\FF)^\op$ corresponding to the unique active morphism $\fromto{I}{\{\xi\}}$ of $\Lambdaup(\FF)$; it is of course egressive in $\XX_\boxtimes$. Hence there is a functor
\[\bigboxtimes_{i\in I}\colon\fromto{\prod_{i\in I}\XX_{S_i}}{\XX_{S}},\]
exact separately in each variable. If $(\phi_i\colon\fromto{S_i}{T_i})_{i\in I}$ is an $I$-tuple of morphisms of $C$ with product $\phi\colon\fromto{S}{T}$ then the square
\begin{equation*}
\begin{tikzpicture}[baseline]
\matrix(m)[matrix of math nodes,
row sep=6ex, column sep=6ex,
text height=1.5ex, text depth=0.25ex]
{\prod_{i\in I}\XX_{T_i} & \XX_{T} \\
\prod_{i\in I}\XX_{S_i} & \XX_{S} \\ };
\path[>=stealth,->,font=\scriptsize]
(m-1-1) edge node[above]{$\bigboxtimes_{i\in I}$} (m-1-2)
edge node[left]{$\prod_{i\in I}\phi_i^\star$} (m-2-1)
(m-1-2) edge node[right]{$\phi^\star$} (m-2-2)
(m-2-1) edge node[below]{$\bigboxtimes_{i\in I}$} (m-2-2);
\end{tikzpicture}
\end{equation*}
commutes.

When $(C,C_{\dag},C^{\dag})$ is cartesian, this structure endows each fiber $\XX_S$ with a symmetric monoidal structure: indeed, for any finite set $I$, we may define
\[\bigotimes_{i\in I}\coloneq\Deltaup^{\ast}\circ\bigboxtimes_{i\in I},\]
where $\Deltaup\colon\fromto{S}{S^I}$ is the diagonal. One sees easily that the commutativity of the square above implies that any functor $\phi^\star$ induced by a morphism $\phi\colon\fromto{S}{T}$ is symmetric monoidal in a natural way. Furthermore, a simple argument demonstrates that the external product $\boxtimes_{i\in I}$ can be recovered from the symmetric monoidal structures along with the pullback functors; for example, $X\boxtimes Y\simeq\pr_1^\star X\otimes\pr_2^\star Y$.

Now it follows from \cite[Cor. 7.3.2.7]{HA} that if $\phi\colon\fromto{S}{T}$ is both ingressive and egressive in $C$, then $\phi_!$ extends to a lax symmetric monoidal functor $\fromto{\XX_S^\otimes}{\XX_T^\otimes}$.
\end{nul}

\begin{lem}\label{potimesisadequate} Suppose $(C,C_{\dag},C^{\dag})$ a left complete disjunctive triple, and suppose
\[p_{\boxtimes}\colon\fromto{\XX_{\boxtimes}}{C_{\times}}\]
a symmetric monoidal Waldhausen bicartesian fibration over $(C,C_{\dag},C^{\dag})$. Then the inner fibration
\[p_{\boxtimes}\colon\fromto{\XX_{\boxtimes}}{C_{\times}}\]
is an adequate inner fibration \cite[Df. 10.3]{M1} for the triple $(C_{\times},(C_{\times})_{\dag},(C_{\times})^{\dag})$ (Nt. \ref{ntn:triplestronCtimes}).
\begin{proof} The only condition of adequate inner fibrations that isn't explicitly part of the definition above is the assertion that for any ingressive morphism $(\phi,\omega):\cofto{(I,X)}{(J,Y)}$ of $C_{\times}$ and for any object $P$ of the fiber $(\XX_{\boxtimes})_{(I,X)}$, there exists a $p_{\boxtimes}$-cocartesian morphism $\fromto{P}{Q}$ covering $(\phi,\omega)$.

So suppose that $(\phi,\omega):\cofto{(I,X)}{(J,Y)}$ is ingressive --- i.e., that $\phi\colon\fromto{J}{I}$ is a bijection and each morphism $\omega_{\phi^{-1}(i)}\colon\fromto{X_{i}}{Y_{\phi^{-1}(i)}}$ is ingressive ---, and suppose that $P$ is an object of $\XX_{\boxtimes}$ that lies over $(I,X)$. Then under the equivalence
\[(\XX_{\boxtimes})_{I}\simeq\prod_{i\in I}\XX_{\{i\}},\]
the object $P$ corresponds to a family $(P_{i})_{i\in I}$ of objects such that $P_{i}$ lies over $X_{i}$ for any $i\in I$. For each $i\in I$, select a $p$-cocartesian edge $\fromto{P_{i}}{Q_{\phi^{-1}(i)}}$ covering $\omega_{\phi^{-1}(i)}$. Now there is an essentially unique morphism $\fromto{P}{Q}$ covering $(\phi,\omega)$ that corresponds under the equivalence above to the edges $\fromto{P_{i}}{Q_{\phi^{-1}(i)}}$, and it is easy to see that it is $p_{\boxtimes}$-cocartesian.
\end{proof}
\end{lem}

If $(C,C_{\dag},C^{\dag})$ is a left complete disjunctive triple, and if $p_{\boxtimes}\colon\fromto{\XX_{\boxtimes}}{C_{\times}}$ a symmetric monoidal Waldhausen bicartesian fibration for $(C,C_{\dag},C^{\dag})$, then our goal is now to equip the unfurling of $\XX$ with the structure of a $A^{\eff}(C)^{\otimes}$-monoidal Waldhausen structure. It will then follow that the corresponding Mackey functor is in fact a commutative Green functor.

\begin{cnstr} Suppose $(C,C_{\dag},C^{\dag})$ a left complete disjunctive triple, and suppose
\[p_{\boxtimes}\colon\fromto{\XX_{\boxtimes}}{C_{\times}}\]
a symmetric monoidal Waldhausen bicartesian fibration over $(C,C_{\dag},C^{\dag})$. Then we define $\Upsilonup(\XX/(C,C_{\dag},C^{\dag}))^{\otimes}$ as the pullback
\begin{equation*}
\Upsilonup(\XX_{\boxtimes}/(C_{\times},(C_{\times})_{\dag},(C_{\times})^{\dag}))\times_{A^{\eff}(C_{\times},(C_{\times})_{\dag},(C_{\times})^{\dag})}A^{\eff}(C,C_{\dag},C^{\dag})^{\otimes}.
\end{equation*}
The inner fibration \cite[Lm. 11.4]{M1}
\[\fromto{\Upsilonup(\XX_{\boxtimes}/(C_{\times},(C_{\times})_{\dag},(C_{\times})^{\dag}))}{A^{\eff}(C_{\times},(C_{\times})_{\dag},(C_{\times})^{\dag})}\]
pulls back to an inner fibration
\begin{equation*}
\Upsilonup(p)^{\otimes}\colon\fromto{\Upsilonup(\XX/(C,C_{\dag},C^{\dag}))^{\otimes}}{A^{\eff}(C,C_{\dag},C^{\dag})^{\otimes}}.
\end{equation*}
We call this the \textbf{\emph{unfurling}} of the symmetric monoidal Waldhausen bicartesian fibration $p_\boxtimes$.
\end{cnstr}

\begin{nul}\label{nul:unwindunfurlsymmonbicart} Suppose, for simplicity, that $(C,C_{\dag},C^{\dag})$ is cartesian. Unwinding the definitions, one sees that the objects of $\Upsilonup(\XX/(C,C_{\dag},C^{\dag}))^{\otimes}$ are precisely the objects of $\XX_{\boxtimes}$. These, in turn, can be thought of as triples $(I,S_I,P_{S_I})$ consiting of a finite set $I$, an $I$-tuple $S_I\coloneq (S_i)_{i\in I}$, and an object $P_{S_I}$ of the fiber
\[(\XX_\otimes)_{S_I}\simeq\prod_{i\in I}\XX_{S_i},\]
which corresponds to an $I$-tuple $(P_{S_i})_{i\in I}$ of objects of the various Waldhausen $\infty$-categories $\XX_{S_i}$. Now a morphism $\fromto{(J,T_J,Q_{T_J})}{(I,S_I,P_{S_I})}$ of the unfurling $\Upsilonup(\XX/(C,C_{\dag},C^{\dag}))^{\otimes}$ can be thought of as the following data:
\begin{enumerate}[(\ref{nul:unwindunfurlsymmonbicart}.1)]
\item a morphism $\phi\colon\fromto{J}{I}$ of $\Lambdaup(\FF)$;
\item a collection of diagrams
\begin{equation*}
\left\{
\begin{tikzpicture}[baseline]
\matrix(m)[matrix of math nodes, 
row sep=3ex, column sep=3ex, 
text height=1.5ex, text depth=0.25ex] 
{&U_{\phi(j)}&\\ 
T_{j}&&S_{\phi(j)},\\}; 
\path[>=stealth,->,font=\scriptsize] 
(m-1-2) edge[->>]  node[above left]{$\tau_{j}$} (m-2-1) 
edge[>->] node[above right]{$\sigma_{\phi(j)}$} (m-2-3);
\end{tikzpicture}
\quad\middle|\quad j\in \phi^{-1}(I)\ 
\right\}
\end{equation*}
of $C$ such that for any $j\in \phi^{-1}(I)$, the morphism $\sigma_j\colon\cofto{U_{\phi(j)}}{S_{\phi(j)}}$ is ingressive, and the morphism $\tau_j\colon\fibto{U_{\phi(j)}}{T_j}$ is egressive; and
\item a collection of morphisms
\[
\left\{
\fromto{\sigma_{\phi(j),!}\tau_{J_i}^\star\left(\bigboxtimes_{j\in J_i}Q_{T_j}\right)}{P_{S_i}}\quad\middle|\quad i\in I
\right\}
\]
in the various $\infty$-categories $\XX_{S_i}$, where $\tau_{J_i}$ is the edge $\fromto{(\{i\},U_i)}{(J_i,T_{J_i})}$ corresponding to the tuple $(\tau_j)_{j\in J_i}$.
\end{enumerate}
\end{nul}

\begin{thm} Suppose $(C,C_{\dag},C^{\dag})$ a left complete disjunctive triple, and suppose
\[p_{\boxtimes}\colon\fromto{\XX_{\boxtimes}}{C_{\times}}\]
a symmetric monoidal Waldhausen bicartesian fibration over $(C,C_{\dag},C^{\dag})$. Then the functor $\Upsilonup(p)^{\otimes}$ exhibits the $\infty$-category $\Upsilonup(\XX/(C,C_{\dag},C^{\dag}))^{\otimes}$ as a $A^{\eff}(C,C_{\dag},C^{\dag})^{\otimes}$-monoidal Waldhausen $\infty$-category.
\begin{proof} We first observe that, in light of \cite[Pr. 11.6]{M1} and Lm. \ref{potimesisadequate}, the functor $\Upsilonup(p)^{\otimes}$ is a cocartesian fibration. Let us check that the composite cocartesian fibration
\[\Upsilonup(\XX/(C,C_{\dag},C^{\dag}))^{\otimes}\to A^{\eff}(C,C_{\dag},C^{\dag})^{\otimes}\to \N\Lambdaup(\FF)\]
exhibits $\Upsilonup(\XX/(C,C_{\dag},C^{\dag}))^{\otimes}$ as a symmetric monoidal $\infty$-category.

To this end, it suffices to show that for any finite set $I$ and any $I$-tuple $S_I\coloneq(S_i)_{i\in I}$ of objects of $C$, the functor
\[\prod_{i\in I}\chi_{i,!}\colon\fromto{(\XX_\boxtimes)_{S_I}\simeq\Upsilonup(\XX/(C,C_{\dag},C^{\dag}))^{\otimes}_{S_I}}{\prod_{i\in I}\Upsilonup(\XX/(C,C_{\dag},C^{\dag}))_{S_i}\simeq\prod_{i\in I}\XX_{S_i}}\]
induced by the cocartesian edges covering the inert maps $\chi_i\colon\fromto{I}{\{i\}_+}$ is an equivalence. But this morphism can be identified with
\[\prod_{i\in I}\left(\id_!\circ\id^\star\circ\bigboxtimes_{i\in\{i\}}\right)\colon\fromto{\prod_{i\in I}\XX_{S_i}}{\prod_{i\in I}\XX_{S_i}},\]
which is homotopic to the identity.

Now for any finite set $J$, a morphism $\fromto{T}{S}$ of $A^{\eff}(C,C_{\dag},C^{\dag})^{\otimes}$ covering the unique active morphism $\fromto{J}{\{\xi\}}$ is represented by a collection of spans
\begin{equation*}
\left\{
\begin{tikzpicture}[baseline]
\matrix(m)[matrix of math nodes, 
row sep=3ex, column sep=3ex, 
text height=1.5ex, text depth=0.25ex] 
{&U&\\ 
T_{j}&&S.\\}; 
\path[>=stealth,->,font=\scriptsize] 
(m-1-2) edge[->>] node[above left]{$\phi_j$} (m-2-1) 
edge[>->] node[above right]{$\psi$} (m-2-3); 
\end{tikzpicture}\quad\middle|\quad j\in J 
\right\}
\end{equation*}
The tensor product functor can therefore be written as
\[\psi_!\circ\phi_J^\star\circ\bigboxtimes_{j\in J}\colon\fromto{\prod_{j\in J}\XX_{T_j}\simeq\XX_T}{\XX_S},\]
which is exact separately in each variable.
\end{proof}
\end{thm}

\noindent In light of Pr. \ref{prp:Waldbicartmon}, we have the following.

\begin{cor} Suppose $(C,C_{\dag},C^{\dag})$ a cartesian disjunctive triple that is either left complete or right complete, and suppose $p_{\boxtimes}\colon\fromto{\XX_{\boxtimes}}{C_{\times}}$ a symmetric monoidal Waldhausen bicartesian fibration over $(C,C_{\dag},C^{\dag})$. Then the cocartesian fibration $\Upsilonup(p)^{\otimes}$ is classified by a Green functor
\[\MM_{p}^{\otimes}\colon\fromto{A^{\eff}(C,C_\dag,C^\dag)^\otimes}{\Wald_\infty^\otimes}.\]
\end{cor}


\section{Equivariant algebraic $K$-theory of group actions} In this section, we answer a question of Akhil Mathew. Namely, for any Waldhausen $\infty$-category $C$ with an action of a finite group $G$, can one form an equivariant algebraic $K$-theory spectrum $K_G(C)$ whose $H$-fixed point spectrum is the algebraic $K$-theory of the homotopy fixed point $\infty$-category $C^{hH}$? Furthermore, can one do this in a lax symmetric monoidal fashion, so that if $C$ is an algebra in Waldhausen $\infty$-categories over an $\infty$-operad $O^{\otimes}$, then $K_G(C)$ is an algebra over $O^{\otimes}$ in $\Mack(\FF_G;\Sp)$? The answer to both of these questions is yes, and our framework makes it an almost trivial matter to see how.

\begin{cnstr} Suppose $G$ a finite group. Let denote by $\FF_G^{\free}\subset\FF_G$ the full subcategory spanned by those finite $G$-sets upon which $G$ acts freely. Observe that $\FF_G^{\free}$ is the finite-coproduct completion of $BG$; that is, it is the free $\infty$-category with finite coproducts generated by $BG$. Consequently, $A^{\eff}(\FF_G^{\free})$ is the free semiadditive $\infty$-category generated by $BG$; that is, evaluation at $G/e$ defines an equivalence
\[\equivto{\Mack(\FF_G^{\free};A)}{\Fun(BG,A)}.\]

At the same time, the subcategory $\FF_G^{\free}\subset\FF_G$ is clearly closed under coproducts, and since $\FF_G^{\free}$ is a sieve in $\FF_G$, it follows that it is stable under pullbacks and binary products as well. Consequently, we obtain a fully faithful inclusion
\[\into{A^{\eff}(\FF_G^{\free})}{A^{\eff}(\FF_G)}.\]
We thus obtain, for any semiadditive $\infty$-category $A$, a corresponding restriction functor
\[\fromto{\Mack(\FF_G;A)}{\Mack(\FF_G^{\free};A)}.\]

If $A$ is an addition presentable, then the restriction functor admits a right adjoint
\[B_G\colon\fromto{\Fun(BG,A)}{\Mack(\FF_G;A)},\]
given by right Kan extension. We shall call this the \emph{Borel functor}, since it assigns to any ``naïve'' $G$-object the corresponding \emph{Borel-equivariant} object.

Applying this when $A=\Wald_{\infty}$ and apply algebraic $K$-theory, we obtain the \emph{algebraic $K$-theory of group actions}:
\[\KK\circ B_G\colon\fromto{\Fun(BG,\Wald_{\infty})}{\Mack(\FF_G;\Sp)}.\]
\end{cnstr}

\begin{prp} The algebraic $K$-theory of group actions extends naturally to a lax symmetric monoidal functor
\[\KK^{\otimes}\circ B_G^{\otimes}\colon\fromto{\Fun(BG,\Wald_{\infty})^{\otimes}}{\Mack(\FF_G;\Sp)^{\otimes}}.\]
for the objectwise symmetric monoidal structure relative to the symmetric monoidal structure on $\Wald_{\infty}$ \cite{K3} and the additivized Day convolution on spectral Mackey functors.
\begin{proof} Since $\KK^{\otimes}$  is lax symmetric monoidal \cite{K3}, it suffices to show that for any presentable semiadditive symmetric monoidal $\infty$-category $E^{\otimes}$, the Borel functor $B_G$ extends to a symmetric monoidal functor
\[B_G^{\otimes}\colon\fromto{\Fun(BG,E)^{\otimes}\simeq\Mack(\FF_G^{\free};A)^{\otimes}}{\Mack(\FF_G;E)^{\otimes}}.\]
This will follow directly from \cite{HA}, once one knows that the restriction functor
\[\fromto{\Mack(\FF_G;E)}{\Fun(BG,E)}\]
extends to a symmetric monoidal functor
\[\fromto{\Mack(\FF_G;E)^{\otimes}}{\Mack(\FF_G^{\free};A)^{\otimes}\simeq\Fun(BG,E)^{\otimes}}.\]
For this, observe that since $\FF_G^{\free}\subset\FF_G$ is stable under binary products, the inclusion
\[\into{A^{\eff}(\FF_G^{\free})}{A^{\eff}(\FF_G)}\]
extends to a symmetric monoidal functor
\[\into{A^{\eff}(\FF_G^{\free})^{\otimes}}{A^{\eff}(\FF_G)^{\otimes}}.\]
It thus suffices to note that for any free finite $G$-set $V$, the subcategory
\[(A^{\eff}(\FF_G^{\free})\times A^{\eff}(\FF_G^{\free})) \times_{A^{\eff}(\FF_G^{\free})} A^{\eff}(\FF_G^{\free})_{/V} \subset (A^{eff}(\FF_G)\times A^{\eff}(\FF_G)) \times_{A^{\eff}(\FF_G)} A^{\eff}(\FF_G)_{/V}\]
is cofinal.
\end{proof}
\end{prp}


\section{Equivariant algebraic $K$-theory of derived stacks} In this section, we construct two symmetric monoidal Waldhausen bicartesian fibrations that extend the following two Waldhausen bicartesian fibrations introduced in \cite[\S D]{M1}:
\begin{itemize}
\item the Waldhausen bicartesian fibration
\[\fromto{\Perf^{\op}\times_{\Shv_{\textit{flat}}}\DM}{\DM}\]
for the left complete disjunctive triple $(\DM,\DM_{\FP},\DM)$ of spectral De\-ligne--Mumford stacks, in which the ingressive morphisms are strongly proper morphisms of finite Tor-amplitude, and all morphisms are egressive \cite[Pr. D.18]{M1}, and
\item the Waldhausen bicartesian fibration
\[\fromto{\Perf^{\op}}{\Shv_{\textit{flat}}}\]
for the left complete disjunctive triple $(\Shv_{\textit{flat}},\Shv_{\textit{flat},\QP},\Shv_{\textit{flat}})$ of flat sheaves in which the ingressive morphisms are the quasi-affine representable and perfect morphisms, and all morphisms are egressive \cite[Pr. D.21]{M1}.
\end{itemize}
These will give algebraic $K$-theory the structure of a commutative Green functor for these two triples.

\begin{nul} To begin, we let
\begin{equation*}
\begin{tikzpicture} 
\matrix(m)[matrix of math nodes, 
row sep=5ex, column sep=4ex, 
text height=1.5ex, text depth=0.25ex] 
{\categ{Mod}^{\otimes}&\QCoh^{\otimes}\\ 
\CAlg^{\mathit{cn}}\times N\Lambdaup(\FF)&\Shv_{\textit{flat}}^{\op}\times N\Lambdaup(\FF)\\}; 
\path[>=stealth,->,font=\scriptsize] 
(m-1-1) edge (m-1-2) 
edge node[left]{$q$} (m-2-1) 
(m-1-2) edge node[right]{$p$} (m-2-2) 
(m-2-1) edge[right hook->] (m-2-2); 
\end{tikzpicture}
\end{equation*}
be a pullback square in which $q$ is the cocartesian fibration of \cite[Th. 4.5.3.1]{HA}, and $p$ is a cocartesian fibration classified by the right Kan extension of the functor that classifies $q$. The objects of $\QCoh^{\otimes}$ can be thought of as triples $(X,I,M_I)$ consisting of a sheaf $X\colon\fromto{\CAlg^{\mathit{cn}}}{\Kan(\kappa_1)}$ for the flat topology, a finite set $I$, and an $I$-tuple $M_I=\{M_i\}_{i\in I}$ of quasicoherent modules $M$ over $X$.
\end{nul}

\begin{nul} We may now pass to the cocartesian $\infty$-operads to obtain a cocartesian fibration of $\infty$-operads
\[p^{\sqcup}\colon\fromto{(\QCoh^{\otimes})^{\sqcup}}{(\Shv_{\textit{flat}}^{\op}\times N\Lambdaup(\FF))^{\sqcup}\simeq(\Shv_{\textit{flat},\times})^{\op}\times_{N\Lambdaup(\FF)}N\Lambdaup(\FF)^{\sqcup}}.\]
Now $\fromto{N\Lambdaup(\FF)^{\sqcup}}{N\Lambdaup(\FF)}$ admits a section that carries any finite set $I$ to the pair $(I,\ast_I)$, where $\ast_I=\{\ast\}_{i\in I}$. Let us pull back $p^{\sqcup}$ along this section to obtain a cocartesian fibration of $\infty$-operads
\[p^{\boxtimes}\colon\fromto{\QCoh^{\boxtimes}\coloneq(\QCoh^{\otimes})^{\sqcup}\times_{N\Lambdaup(\FF)^{\sqcup}}N\Lambdaup(\FF)}{(\Shv_{\textit{flat},\times})^{\op}}.\]
\end{nul}

\begin{nul} Passing to opposites, we obtain a functor
\[\fromto{(\QCoh^{\op})_\boxtimes\coloneq(\QCoh^{\boxtimes})^{\op}}{\Shv_{\textit{flat},\times}}\]
which
\begin{itemize}
\item restricts to a symmetric monoidal Waldhausen bicartesian fibration
\[\fromto{(\QCoh^{\op})_\boxtimes\times_{\Shv_{\textit{flat},\times}}\DM_{\times}}{\DM_{\times}}\]
that extends the Waldhausen bicartesian fibration of \cite[Pr. D.10]{M1} for the disjunctive triple of spectral Deligne--Mumford stacks, in which the ingressive morphisms are relatively scalloped, and all morphisms are egressive, and
\item gives a symmetric monoidal Waldhausen bicartesian fibration
\[\fromto{(\QCoh^{\op})_\boxtimes}{\Shv_{\textit{flat},\times}}\]
that extends the Waldhausen bicartesian fibration of \cite[Pr. D.13]{M1} for the disjunctive triple of flat sheaves, in which the ingressive morphisms are quasi-affine representable, and all morphisms are egressive.
\end{itemize}
\end{nul}

\begin{nul} At last, restricting to perfect modules, we obtain the desired symmetric monoidal Waldhausen bicartesian fibrations
\[\fromto{(\Perf^{\op})_{\boxtimes}\times_{(\Shv_{\textit{flat}})_{\times}}\DM_{\times}}{\DM_{\times}}\]
for $(\DM,\DM_{\FP},\DM)$ and
\[\fromto{(\Perf^{\op})_{\boxtimes}}{(\Shv_{\textit{flat}})_{\times}}\]
for $(\Shv_{\textit{flat}},\Shv_{\textit{flat},\QP},\Shv_{\textit{flat}})$.
\end{nul}

Now, passing to the unfurling, we obtain the following pair of results.
\begin{prp} The Mackey functor
\[\MM_{\DM}\colon\fromto{A^{\eff}(\DM,\DM_{\FP},\DM)}{\Wald_{\infty}}\]
of \cite[Cor. D.18.1]{M1} admits a natural structure of a commutative Green functor $\MM_{\DM}^{\otimes}$. In particular, the algebraic $K$-theory of spectral Deligne--Mumford stacks is naturally a commutative spectral Green functor for $(\DM,\DM_{\FP},\DM)$.
\end{prp}

\begin{prp} The Mackey functor
\[\MM_{\Shv_{\textit{flat}}}\colon\fromto{A^{\eff}(\Shv_{\textit{flat}},\Shv_{\textit{flat},\QP},\Shv_{\textit{flat}})}{\Wald_{\infty}}\]
of \cite[Cor. D.21.1]{M1} admits a natural structure of a commutative Green functor $\MM_{\Shv_{\textit{flat}}}^{\otimes}$. In particular, the algebraic $K$-theory of flat sheaves is naturally a commutative spectral Green functor for $(\Shv_{\textit{flat}},\Shv_{\textit{flat},\QP},\Shv_{\textit{flat}})$.
\end{prp}

\begin{cnstr}\label{cnstr:Galoisequivarianteinftyalg} Suppose $X$ a spectral Deligne--Mumford stack. As in \cite[Nt. D.23]{M1}, we denote by $\FEt(X)$ the subcategory of $\DM_{/X}$ whose objects are finite \cite[Df. 3.2.4]{DAGXII} and \'etale morphisms $\fromto{Y}{X}$ and whose morphisms are finite and \'etale morphisms over $X$. Observe that the fiber product $-\times_X-$ endows $\FEt(X)$ with the structure of a cartesian disjunctive $\infty$-category. We will abuse notation and write $A^{\eff}(X)^{\otimes}$ for the symmetric monoidal effective Burnside $\infty$-category of $\FEt(X)$.

Now the inclusion
\[\into{(\FEt(X),\FEt(X),\FEt(X))}{(\DM,\DM_{\FP},\DM)}\]
is clearly a morphism of cartesian disjunctive triples, whence one can restrict the commutative Green functor $\MM_{\DM}^{\otimes}$ above along the morphism of $\infty$-operads
\[\fromto{A^{\eff}(X)^{\otimes}}{A^{\eff}(\DM,\DM_{\FP},\DM)^{\otimes}}\]
to a commutative Green functor
\[\MM_X\colon\fromto{A^{\eff}(X)^{\otimes}}{\Wald_\infty^{\otimes}}.\]

Now if $X$ is (say) a connected, noetherian scheme, then a choice of geometric point $x$ of $X$ gives rise to an equivalence
\[A^{\eff}(\pi_1^{\et}(X,x))^{\otimes}\simeq A^{\eff}(X)^{\otimes}.\]
Applying algebraic $K$-theory, we obtain a commutative spectral Green functor for the étale fundamental group:
\begin{equation*}
\KK^{\otimes}_{\pi_1^{\et}(X,x)}(X)\colon\fromto{A^{\eff}(\pi_1^{\et}(X,x))^{\otimes}}{\Sp^{\otimes}}.
\end{equation*}
This commutative Green functor deserves the handle \emph{Galois-equivariant algebraic $K$-theory}.
\end{cnstr}


\section{An equivariant Barratt--Priddy--Quillen Theorem}

\begin{ntn} In this section, suppose $(C,C_\dag,C^\dag)$ a cartesian disjunctive triple.
\end{ntn}

\begin{rec} \label{cnstr:RCCdagCdag} Recall \cite[Df. 13.5]{M1} that $\RR(C)\subset\Fun(\Deltaup^2/\Deltaup^{\{0,2\}},C)$ is the full subcategory spanned by those retract diagrams
\begin{equation*}
S_0\to S_1\to S_0;
\end{equation*}
such that the morphism $\fromto{S_0}{S_1}$ is a summand inclusion. We endow $\RR(C)$ with the structure of a pair in the following manner. A morphism $\fromto{T}{S}$ will be declared ingressive just in case $\fromto{T_0}{S_0}$ is an equivalence, and $\fromto{T_1}{S_1}$ is a summand inclusion. Write $p$ for the functor $\fromto{\RR(C)}{C}$ given by evaluation at the vertex $0=2$:
\begin{equation*}
\goesto{[S_0\to S_1\to S_0]}{S_0}.
\end{equation*}
Recall also that $\RR(C,C_{\dag},C^{\dag})\subset\RR(C)$ is the full subcategory spanned by those objects
\[S\colon\fromto{\Deltaup^2/\Deltaup^{\{0,2\}}}{C}\]
such that for any complement $\into{S'_0}{S_1}$ of the summand inclusion $\into{S_0}{S_1}$,
\begin{enumerate}[(\ref{cnstr:RCCdagCdag}.1)]
\item the essentially unique morphism $\fromto{S'_0}{1}$ to the terminal object of $C$ is egressive, and
\item the composite $S'_0\to S_1\to S_0$ is ingressive.
\end{enumerate}
We endow $\RR(C,C_{\dag},C^{\dag})$ with the pair structure induced by $\RR(C)$. We will abuse notation by denoting the restriction of the functor $p\colon\fromto{\RR(C)}{C}$ to the subcategory $\RR(C,C_{\dag},C^{\dag})\subset\RR(C)$ again by $p$.

We proved in \cite[Th. 13.11]{M1} that $p$ is a Waldhausen bicartesian fibration over $(C,C_\dag,C^\dag)$.
\end{rec}

\begin{cnstr} Recall that an object of the $\infty$-category $\RR(C,C_{\dag},C^{\dag})_\times$ can be described as pairs $(I,X)$ consisting of a finite set $I$ and a collection $X=\{X_i\ |\ i\in I\}$ of objects of $\RR(C,C_{\dag},C^{\dag})$ indexed by the elements of $I$. Accordingly, a morphism $\fromto{(I,X)}{(J,Y)}$ of $\RR(C,C_{\dag},C^{\dag})_\times$ can be described as a map $\fromto{J}{I_{+}}$ of finite sets and a collection
\[\left\{\fromto{X_i}{\prod_{j\in J_i}Y_j}\quad\middle|\quad i\in I\right\}\]
of morphisms of $\RR(C,C_{\dag},C^{\dag})$.

We now define a subcategory $\RR(C,C_{\dag},C^{\dag})_{\boxtimes}\subset\RR(C,C_{\dag},C^{\dag})_\times$ that contains all the objects. A morphism $\fromto{(I,X)}{(J,Y)}$ of $\RR(C,C_{\dag},C^{\dag})_\times$ is a morphism of $\RR(C,C_{\dag},C^{\dag})_{\boxtimes}$ if and only if, for every $i\in I$, every nonempty proper subset $K_i\subset J_i$, and every choice of a complement $\into{Y'_{\mskip-3mu j,0}}{Y_{\mskip-3mu j,1}}$ of the summand inclusion $\into{Y_{\mskip-3mu j,0}}{Y_{\mskip-3mu j,1}}$, the square
\begin{equation*}
\begin{tikzpicture} 
\matrix(m)[matrix of math nodes, 
row sep=4ex, column sep=4ex, 
text height=1.5ex, text depth=0.25ex] 
{\varnothing&X_{i,1}\\ 
\prod_{j\in K_i}Y_{\mskip-3mu j,0}\times\prod_{j\in J_i\setminus K_i}Y'_{\mskip-3mu j,0}&\prod_{j\in J_i}Y_{\mskip-3mu j,1},\\}; 
\path[>=stealth,->,font=\scriptsize] 
(m-1-1) edge (m-1-2) 
edge (m-2-1) 
(m-1-2) edge (m-2-2) 
(m-2-1) edge (m-2-2); 
\end{tikzpicture}
\end{equation*}
in which $\varnothing$ is initial and the bottom morphism is the obvious summand inclusion, is a pullback.

Let us endow this $\infty$-category with a pair structure in the following manner. We declare a morphism $\fromto{(I,X)}{(J,Y)}$ of $\RR(C,C_{\dag},C^{\dag})_{\boxtimes}$ to be ingressive just in case the map $\fromto{J}{I_{+}}$ represents an isomorphism in $\Lambdaup(\FF)$, and, for every $i\in I$, the map $\fromto{X_i}{Y_{\phi(i)}}$ of $\RR(C,C_{\dag},C^{\dag})$ is ingressive.
\end{cnstr}

The following is now immediate.

\begin{prp} The functor
\[p_\boxtimes\colon\fromto{\RR(C,C_{\dag},C^{\dag})_\boxtimes}{C_\times}\]
given by evaluation at $0=2$ in $\Deltaup^2/\Deltaup^{\{0,2\}}$exhibits $\RR(C,C_{\dag},C^{\dag})$ as a symmetric monoidal Waldhausen bicartesian fibration over $(C,C_{\dag},C^{\dag})$.
\end{prp}

\begin{cnstr} Now we may the unfurling construction of \cite[\S 11]{M1} to the symmetric monoidal Waldhausen bicartesian fibration $p_\boxtimes$ to obtain an $A^{\eff}(C,C_{\dag},C^{\dag})^{\otimes}$-monoidal Waldhausen $\infty$-category (in the sense of \cite{K3})
\begin{equation*}
\Upsilonup(p)^\otimes\colon\fromto{\Upsilonup(\RR(C,C_{\dag},C^{\dag})/(C,C_{\dag},C^{\dag}))^{\otimes}}{A^{\eff}(C,C_{\dag},C^{\dag})^{\otimes}}.
\end{equation*}
As we've demonstrated, $\Upsilonup(p)^\otimes$ is classified by an $E_{\infty}$ Green functor
\begin{equation*}
\MM^{\otimes}_{p}\colon\fromto{A^{\eff}(C,C_{\dag},C^{\dag})^{\otimes}}{\Wald_{\infty}^{\otimes}}
\end{equation*}
whose underlying functor is the Mackey functor
\begin{equation*}
\MM_{p}\colon\fromto{A^{\eff}(C,C_{\dag},C^{\dag})}{\Wald_{\infty}}
\end{equation*}
corresponding to the unfurling of the Waldhausen bicartesian fibration
\[\fromto{\RR(C,C_{\dag},C^{\dag})}{C}\]
over $(C,C_{\dag},C^{\dag})$.
\end{cnstr}

In \cite{K3}, we demonstrated that algebraic $K$-theory lifts in a natural fashion to a morphism of $\infty$-operads, whence we may contemplate the commutative Green functor
\begin{equation*}
\KK^{\otimes}\circ\MM^{\otimes}_{p}\colon\fromto{A^{\eff}(C,C_{\dag},C^{\dag})^\otimes}{\Sp^\otimes}.
\end{equation*}
Observe that by \cite[Th. 13.12]{M1}, the underlying Mackey functor
\begin{equation*}
\SS_{(C,C_{\dag},C^{\dag})}\coloneq\KK\circ\MM_{p}
\end{equation*}
of $\KK^{\otimes}\circ\MM^{\otimes}_{p}$ is the spectral Burnside Mackey functor for $(C,C_{\dag},C^{\dag})$, as defined in \cite[Df. 8.1]{M1}. In particular, it is unit for the symmetric monoidal $\infty$-category $\Mack(C,C_{\dag},C^{\dag};\Sp)$, which of course admits an essentially unique $E_\infty$ structure. Consequently, we deduce the following.

\begin{thm}[Equivariant Barratt--Priddy--Quillen]\label{thm:ebpq} The Green functor $\KK^{\otimes}\circ\MM^{\otimes}_{p}$ is the spectral Burnside Green functor $\SS_{(C,C_{\dag},C^{\dag})}$.
\end{thm}

\noindent Of course, this result directly implies the original Barratt--Priddy--Quillen Theorem, which states that the algebraic $K$-theory of the ordinary Waldhausen category $\FF_\ast$ of pointed finite sets (in which the cofibrations are the monomorphisms) is the sphere spectrum $\SS$. Furthermore, the essentially unique $E_\infty$ structure on $\SS$ is induced by the smash product of pointed finite sets.


\section{A brief epilogue about the Theorems of Guillou--May}

Suppose $G$ a finite group. Write $\categ{OrthSp}_G$ for the underlying $\infty$-category of the relative category of orthogonal $G$-spectra. The Equivariant Barratt--Priddy--Quillen Theorem of Guillou--May \cite{guillou2} provides a similar description in $\categ{OrthSp}_G$ of certain mapping spectra. Note that this is not \emph{a priori} related to Th. \ref{thm:ebpq} when $C=\FF_G$. Nevertheless, a suitable comparison theorem (which of course Guillou--May provide in \cite{guillou3}) offers an implication.

On the other hand, the proof of our result here, combined with work from our forthcoming book \cite{BDGNS1}, will allow us to reprove, using entirely different methods, the comparison result of Guillou--May. Indeed, if we can extend the functor $\Sigma^{\infty}_+\colon\fromto{\FF_G}{\categ{OrthSp}_G}$ to a suitable functor $\fromto{A^{\eff}(\FF_G)}{\categ{OrthSp}_G}$, then the Equivariant Barratt--Priddy--Quillen Theorem above and the Schwede--Shipley theorem \cite{MR2003g:55034} together will imply the result of Guillou--May \cite{guillou3} providing the equivalence
\[\Sp^G\simeq\categ{OrthSp}_G.\]

It is, however, difficult to construct the desired functor $\fromto{A^{\eff}(\FF_G)}{\categ{OrthSp}_G}$ directly, as this involves nontrivial homotopy coherence problems. However, in the language of $G$-equivariant $\infty$-category theory, which we develop in the forthcoming \cite{BDGNS1} provides a universal property for the $G$-equivariant effective Burnside $\infty$-category. This will provide us with the desired functor, and we will easily deduce the desired equivalence as a corollary.

\bibliographystyle{amsplain}
\bibliography{kthy}

\end{document}